\documentclass[12pt]{article}

\usepackage{amsmath,amssymb,amsthm}
\usepackage{xcolor}
\definecolor{allrefcolors}{rgb}{0,0.5,0.4}
\usepackage[pagebackref,linktocpage=true,colorlinks=true,allcolors=allrefcolors,bookmarksopen]{hyperref}
\usepackage[utf8]{inputenc}

\newtheorem{theorem}{Theorem}[section]
\newtheorem{lemma}[theorem]{Lemma}

\newtheorem{proposition}[theorem]{Proposition}

\theoremstyle{definition}
\newtheorem{definition}[theorem]{Definition}
\theoremstyle{remark}
\newtheorem{remark}[theorem]{Remark}

\numberwithin{equation}{section}

\newcommand{\R}{\mathbb R}
\newcommand{\C}{\mathbb C}
\newcommand{\Z}{\mathbb Z}
\newcommand{\W}{\mathfrak W}
\newcommand{\p}{\mathbf p}
\newcommand{\vol}{\operatorname{vol}}
\newcommand{\Diff}{\operatorname{Diff}}
\newcommand{\crit}{\operatorname{crit}}
\newcommand{\im}{\operatorname{im}}
\newcommand{\id}{\operatorname{id}}
\newcommand{\sm}{\mathrm{sm}}

\newcommand{\std}{\mathrm{std}}

\renewcommand{\Re}{\operatorname{Re}}
\renewcommand{\Im}{\operatorname{Im}}
\newcommand{\comment}[1]{}

\def\ol{\overline}
\def\rst#1{|_{#1}}
\def\der_#1#2{\frac{ \partial #2 }{ \partial #1 }}
\def\OO{\mathcal{O}}
\def\eps{\epsilon}

\begin{document}

\title{Existence of Lefschetz fibrations on Stein and Weinstein domains}

\author{Emmanuel Giroux\\John Pardon}

\date{17 July 2015; Revised 7 April 2016}

\maketitle

\begin{abstract}
We show that every Stein or Weinstein domain may be presented (up to
deformation) as a Lefschetz fibration over the disk. The proof is an application
of Donaldson's quantitative transversality techniques.
\end{abstract}


\section{Introduction}

In this paper, we prove the existence of \emph{Lefschetz fibrations} (certain 
singular fibrations with Morse-type singularities) on \emph{Stein domains} 
(from complex geometry) and on \emph{Weinstein domains} (from symplectic 
geometry).  These two results are linked (in fact, logically so) by 
the close relationship between Stein and Weinstein structures established 
in the book by Cieliebak--Eliashberg \cite{cieliebakeliashberg} building 
on earlier work of Eliashberg \cite{eliashbergstein}.  Nevertheless, 
they can be understood independently from either a purely complex geometric 
viewpoint or from a purely symplectic viewpoint.

\subsection{Lefschetz fibrations on Stein domains}

We begin by explaining our results for Stein domains.

\begin{definition}
A real-valued function $\phi$ on a complex manifold $V$ is called 
\emph{$J$-convex} (or \emph{strictly plurisubharmonic}) iff 
$(id'd''\phi)(v,Jv)>0$ 
for every nonzero (real) tangent vector $v$.
\end{definition}

\begin{definition}
A \emph{Stein manifold} is a complex manifold $V$ which admits a smooth
exhausting $J$-convex function $\phi: V \to \R$.
\end{definition}

\begin{definition} \label{steindomaindef}
A \emph{Stein domain} is a compact complex manifold with boundary 
$V$ which admits a smooth $J$-convex function $\phi: V \to \R$ with $\partial V 
= \{\phi = 0\}$ as a regular level set.

(For us, a complex manifold with boundary (or corners) shall mean one equipped 
with a germ of (codimension zero) embedding into an (open) complex manifold. A 
holomorphic function on a complex manifold with boundary (or corners) is one 
which extends holomorphically to an open neighborhood in the ambient (open) 
complex manifold.)
\end{definition}

For example, if $\ol V$ is a Stein manifold with smooth exhausting $J$-convex 
function $\phi: \ol V \to \R$ with $\{\phi = 0\}$ as a regular level set, then 
$V := \{\phi \le 0\}$ is a Stein domain. In fact, it is not hard to see that 
every Stein domain is of this form.

\begin{definition} \label{steinLF}
Let $D^2 \subseteq \C$ denote the closed unit disk. A \emph{Stein Lefschetz 
fibration} is a holomorphic map $\pi: V \to D^2$ where $V$ is a compact complex 
manifold with corners, such that:
\begin{itemize}
\item (Singular fibration)
The map $\pi$ is a (smooth) fibration with manifold with boundary fibers, except
for a finite number of critical points $\crit(\pi)$ in the interior of $V$.
\item (Non-degenerate critical points)
Near each critical point $p \in \crit(\pi)$, there are local holomorphic
coordinates in which $\pi$ is given by $(z_1, \ldots, z_n) \mapsto \pi(p) + 
\sum_{i=1}^nz_i^2$ (according to the complex Morse lemma, this holds iff the 
complex Hessian at $p$ is non-degenerate). Furthermore, all critical values 
are dinstinct.
\item (Stein fibers)
There exists a $J$-convex function $\phi: V \to \R$ with $\partial_h V = \{\phi
= 0\}$ as a regular level set, where $\partial_h V := \bigcup_{p \in D^2} 
\partial (\pi^{-1}(p))$ denotes the ``horizontal boundary'' of $V$.
\end{itemize}
Note that the boundary of $V$ is the union of the horizontal boundary 
$\partial_h V$ and the ``vertical boundary'' $\partial_v V := 
\pi^{-1}(\partial D^2)$, whose intersection $\partial_h V \cap \partial_v V$ 
is the corner locus.

The total space $V$ of any Stein Lefschetz fibration may be smoothed out to 
obtain a Stein domain $V^\sm$, unique up to deformation. Specifically, for any 
function $g: \R_{<0} \to \R$ satisfying $g'>0$, $g''>0$, and $\lim_{x \to 0^-} 
g(x) = \infty$, the function $\Phi_g := g(|\pi|^2-1) + g(\phi)$ is an exhausting
$J$-convex function on $V^\circ$. Moreover, the critical locus of $\Phi_g$ stays
away from $\partial V$ as $g$ varies in any compact family (this follows from
the obvious inclusion $\crit (\Phi_g) \subseteq \bigcup_{p \in D^2} \crit (\phi
\rst{ \pi^{-1}(p) })$ and the fact that the latter is a compact subset of $V
\setminus \partial_h V$). As a result, the sublevel set $\{\Phi_g \leq M\}$ is a
Stein domain which, up to deformation, is independent of the choice of $g$ and
the choice of $M$ larger than all critical values of $\Phi_g$. We denote this
(deformation class of) Stein domain by $V^\sm$, which, of course, depends not
only on $V$, but also on $\pi$.
\end{definition}

The simplest (and weakest) version of our existence result is the following.

\begin{theorem} \label{steinexistence}
Let $V$ be a Stein domain. There exists a (Stein) Lefschetz fibration $\pi: V'
\to D^2$ with $(V')^\sm$ deformation equivalent to $V$.
\end{theorem}

Deformation is meant in the sense of a real $1$-parameter family of Stein 
domains. The nature of the deformation required is made explicit by considering
the following stronger version of our existence result.

\begin{theorem} \label{refinedsteinexistence}
Let $V$ be a Stein domain. For every sufficiently large real number $k$, there 
exists a holomorphic function $\pi: V \to \C$ such that:
\begin{itemize}
\item
For $|\pi(p)| \ge 1$, we have $d\log\pi(p) = k \cdot d'\phi(p) + O(k^{1/2})$.
\item
For $|\pi(p)| \le 1$ and $p \in \partial V$, we have $d\pi(p) \rst \xi \ne 0$.
\end{itemize}
We may, in addition, require that $\pi^{-1}(D^2)$ contain any given compact 
subset of $V^\circ$.
\end{theorem}

Theorem \ref{steinexistence} follows from Theorem 
\ref{refinedsteinexistence} by smoothing out the deformation of Stein domains $\{ \pi^{-1}
(D^2_r) \}_{1 \le r < \infty}$ (this argument is given in detail in 
\S\ref{steinLFsec}). Theorem \ref{refinedsteinexistence} is a corollary of the 
following, which is the main technical result of the paper.

\begin{theorem} \label{quantitativesteinexistence}
Let $\ol V$ be a Stein manifold, equipped with a smooth exhausting $J$-convex
function $\phi: \ol V \to \R$. For every sufficiently large real number $k$,
there exists a holomorphic function $f: \ol V \to \C$ such that:
\begin{itemize}
\item
$|f(p)| \le e^{\frac 12k\phi(p)}$ for $p \in \{\phi \le 1\}$.
\item
$|f(p)| + k^{-1/2} |df(p) \rst \xi| > \eta$ for $p \in \{\phi = 0\}$ ($df$
measured in the metric induced by $\phi$).
\end{itemize}
where $\xi$ denotes the Levi distribution on $\{\phi = 0\} \subseteq \ol V$, and
$\eta>0$ is a constant depending only on the dimension of $\ol V$.
\end{theorem}

To prove Theorem \ref{refinedsteinexistence} (for $V := \{\phi \le 0\}$) from 
Theorem \ref{quantitativesteinexistence}, we take $\pi$ to be (a small 
perturbation of) $\eta^{-1} \cdot f$, which works once $k$ is sufficiently large 
(the details of this argument are given in \S\ref{steinLFsec}). To prove Theorem 
\ref{quantitativesteinexistence}, we use methods introduced by Donaldson \cite
{donaldsonI} (this proof occupies \S\ref{complexgeosec}--\ref{qtranssec}). A
closely related result was obtained by Mohsen \cite{mohsenpreprint} also using
Donaldson's techniques.

\subsection{Lefschetz fibrations on Weinstein domains}

Next, we turn to our result for Weinstein domains.

\begin{definition}
A \emph{Weinstein domain} is a compact symplectic manifold with boundary $(W,
\omega)$ equipped with a $1$-form $\lambda$ satisfying $d\lambda = \omega$ and a
Morse function $\phi: W \to \R$ which has $\partial W = \{\phi = 0\}$ as a
regular level set and for which $X_\lambda$ (defined by $\omega(X_\lambda,\cdot)
= \lambda$) is gradient-like.
\end{definition}

\begin{definition}
An \emph{abstract Weinstein Lefschetz fibration} is a tuple
\begin{equation*}
W = (W_0; L_1, \ldots, L_m)
\end{equation*}
consisting of a Weinstein domain $W_0^{2n-2}$ (the ``central
fiber'') along with a finite sequence of exact parameterized\footnote
{Parameterized shall mean equipped with a diffeomorphism $S^{n-1} \xrightarrow
\sim L$ defined up to precomposition with elements of $O(n)$.}
Lagrangian spheres $L_1, \ldots, L_m \subseteq W_0$ (the ``vanishing cycles'').

From any abstract Weinstein Lefschetz fibration $W = (W_0; L_1, \ldots, L_m)$, 
we may construct a Weinstein domain $|W|$ (its ``total space'') by attaching 
critical Weinstein handles to the stabilization $W_0 \times D^2$ along
Legendrians $\Lambda_j \subseteq W_0 \times S^1 \subseteq \partial (W_0 \times 
D^2)$ near $2\pi j/m \in S^1$ obtained by lifting the exact Lagrangians $L_j$.  
We give this construction in detail in \S\ref{weinsteinLFsec}.
\end{definition}

We will prove the following existence result.

\begin{theorem} \label{weinsteinexistence}
Let $W$ be a Weinstein domain.  There exists an abstract Weinstein Lefschetz 
fibration $W' = (W_0; L_1, \ldots, L_m)$ whose total space $\left|W'\right|$ is 
deformation equivalent to $W$.
\end{theorem}

Deformation is meant in the sense of a $1$-parameter family of Weinstein 
domains, but where the requirement that $\phi$ be Morse is relaxed to allow 
birth death critical points. Theorem \ref{weinsteinexistence} is deduced from 
Theorem \ref{steinexistence} using the existence theorem for Stein structures 
on Weinstein domains proved by Cieliebak--Eliashberg \cite[Theorem 1.1(a)]
{cieliebakeliashberg}. The main step is thus to show that a Stein Lefschetz
fibration $\pi: V \to D^2$ naturally gives rise to an abstract Weinstein
Lefschetz fibration whose total space is deformation equivalent to $V^\sm$ (the 
details of this argument are given in \S\ref{weinsteinLFsec}).

In current work in progress, we hope to apply Donaldson's techniques directly 
in the Weinstein setting to produce on any Weinstein domain $W$ an approximately
holomorphic function $f : W \to \C$ satisfying conditions similar to those in 
Theorem \ref{quantitativesteinexistence}, and thus give a proof of Theorem 
\ref{weinsteinexistence} which does not appeal to the existence of a compatible 
Stein structure.

Given Theorem \ref{weinsteinexistence}, it is natural to ask whether every 
deformation equivalence between the total spaces of two abstract Weinstein 
Lefschetz fibrations is induced by a finite sequence of moves of some simple 
type. Specifically, applying any of the following operations to an abstract 
Weinstein Lefschetz fibration preserves the total space up to canonical 
deformation equivalence, and it is natural to ask whether they are enough.
\begin{itemize}
\item(Deformation)
Simultaneous Weinstein deformation of $W_0$ and exact Lagrangian isotopy of
$(L_1, \ldots, L_m)$.
\item(Cyclic permutation)
Replace $(L_1, \ldots, L_m)$ with $(L_2, \ldots, L_m, L_1)$.
\item(Hurwitz moves)
Let $\tau_L$ denote the symplectic Dehn twist around $L$, and replace $(L_1,
\ldots, L_m)$ with either $(L_2, \tau_{L_2} L_1, L_3, \ldots, L_m)$ or
$(\tau_{L_1}^{-1} L_2, L_1, L_3, \ldots, L_m)$.
\item(Stabilization)
For a parameterized Lagrangian disk $D^{n-1} \hookrightarrow W_0$ with 
Legendrian boundary $S^{n-2} = \partial D^{n-1} \hookrightarrow \partial W_0$ 
such that $0 = [\lambda_0] \in H^1 (D^{n-1}, \partial D^{n-1})$, replace $W_0$ 
with $\tilde W_0$, obtained by attaching a Weinstein handle to $W_0$ along 
$\partial D^{n-1}$, and replace $(L_1, \ldots, L_m)$ with $(\tilde L, L_1, 
\ldots, L_m)$, where $\tilde L \subseteq \tilde W_0$ is obtained by gluing 
together $D^{n-1}$ and the core of the handle.
\end{itemize}
It would be very interesting if the methods of this paper could be brought to 
bear on this problem as well.

\begin{remark}
The reader is likely to be familiar with more geometric notions of symplectic 
Lefschetz fibrations (e.g., as in Seidel \cite[\S 15d]{seidel} or 
Bourgeois--Ekholm--Eliashberg \cite[\S 8.1]{bourgeoisekholmeliashberg} and the 
references therein), and may prefer these to the notion of an abstract 
Weinstein Lefschetz fibration used to state Theorem \ref{weinsteinexistence}. 
We believe, though, that the reader wishing to construct a symplectic Lefschetz 
fibration in their preferred setup with the same total space as a given 
abstract Weinstein Lefschetz fibration will have no trouble doing so (e.g., see 
Seidel \cite[\S 16e]{seidel}).
\end{remark}

Seidel \cite{seidel,seidelainftysubalg,seidelshhh,seidelainftystruct} has 
developed powerful methods for calculations in and of Fukaya categories coming 
from Lefschetz fibrations, in particular relating the Fukaya category of the 
total space to the vanishing cycles and the Fukaya category of the central 
fiber.  Our existence result shows that these methods are applicable to any 
Weinstein domain.  We should point out, however, that, while our proof of 
existence of Lefschetz fibrations is in principle effective, it does not 
immediately lead to any practical way of computing a Lefschetz presentation of 
a given Weinstein manifold.

\subsection{Remarks about the proof}

We outline briefly the proof of Theorem \ref{quantitativesteinexistence} (the 
main technical result of the paper), which occupies \S\ref{complexgeosec}--\ref
{qtranssec}. As mentioned earlier, the proof is an application of Donaldson's
quantitative transversality techniques, first used to construct symplectic
divisors inside closed symplectic manifolds \cite{donaldsonI} (somewhat similar 
ideas appeared earlier in Cheeger--Gromov \cite{cheegergromov}).

The $J$-convex function $\phi: \ol V \to\R$ determines a positive line bundle
$L$ on $\ol V$. We consider the high tensor powers $L^k$ of this positive line
bundle. Using $L^2$-methods of Hörmander \cite{hormander} and
Andreotti--Vesentini \cite{andreottivesentini}, one may construct ``peak
sections'' of $L^k$, that is, holomorphic sections $s: \ol V \to L^k$ which are 
``concentrated'' over the ball of radius $k^{-1/2}$ centered at any given point 
$p_0 \in V := \{\phi \le 0\}$ and have decay $|s(p)| = O (e^{-\eps \cdot k \cdot
d(p,p_0)^2})$ for $p \in \{\phi \le 1\}$.

Donaldson introduced a remarkable method to, given enough localized holomorphic 
sections, construct a linear combination $s: \ol V \to L^k$ which satisfies,
quantitatively, any given holomorphic transversality condition which is generic. 
The key technical ingredient for Donaldson's construction is a suitably 
quantitative version of Sard's theorem, and this step was simplified 
considerably by Auroux \cite{aurouxremark}. The function $f$ asserted to exist 
in Theorem \ref{quantitativesteinexistence} is simply the quotient of such a 
quantitatively transverse section $s: \ol V \to L^k$ by a certain tautological 
section $\text{``$1$''}: \ol V \to L^k$.

We take advantage of the fact that we are in the holomorphic category by working
with genuinely holomorphic functions, instead of the approximately holomorphic
functions which are the standard context of Donaldson's techniques. This allows 
us to use simplified arguments at various points in the proof, and this is the
reason for our passage from the Weinstein setting to the Stein setting. It is
not clear whether one should expect to be able to generalize our arguments to
apply directly to Weinstein manifolds.

Note that in most applications of quantitative transversality techniques in 
symplectic/contact geometry, the result in the integrable case requires 
only generic transversality, and the passage from integrable to non-integrable 
$J$ is what necessitates quantitative transversality.  Here, quantitative 
transversality is needed in both the integrable and non-integrable settings 
(although indeed, one would need more quantitative transversality in the 
non-integrable case).

Besides Donaldson's original paper \cite{donaldsonI}, which is the best place to
first learn the methods introduced there, let us mention a few other papers
where approximately holomorphic techniques have been used to obtain results
similar to Theorem \ref{weinsteinexistence}. In addition to constructing
symplectic divisors \cite{donaldsonI}, Donaldson also constructed Lefschetz
pencils on closed symplectic manifolds \cite{donaldsonII}. Auroux \cite{aurouxI,
aurouxII} further generalized and refined Donaldson's techniques to
$1$-parameter families of sections and to high twists $E \otimes L^k$ of a given
Hermitian vector bundle $E$. In particular, he showed that Donaldson's
symplectic divisors are all isotopic for fixed sufficiently large $k$, and that 
symplectic four-manifolds can be realized as branched coverings of $\C P^2$.
Ibort--Martínez-Torres--Presas \cite{ibortmartineztorrespresas} obtained 
analogues for contact manifolds of Donaldson's and Auroux's results, and these
were used in \cite{giroux} to construct open books on contact manifolds
in any dimension. 
Mohsen \cite{mohsenthesis,mohsenpreprint} extended the techniques of Donaldson 
and Auroux to construct sections whose restrictions to a given submanifold 
satisfy certain quantitative transversality conditions. He also showed that 
this result implies both the uniqueness theorem of Auroux on symplectic 
divisors and the contact theorem of Ibort--Martínez--Presas.  His main 
observation is that the quantitative Sard theorem applies to real (not just to 
complex) polynomials. This plays an important role in the present work; it
makes it possible to obtain quantitative transversality for the restriction of 
a holomorphic section to a real hypersurface.

\subsection{Acknowledgements}

We wish to thank Yasha Eliashberg, Jean-Paul Mohsen, and Paul Seidel for 
helpful discussions and encouragement.  We thank Sylvain Courte for pointing 
out an error in an earlier version of the proof of Lemma \ref
{elementarycobordismisweinsteinhandle}, and we thank the referee for their 
careful reading.

This collaboration started after E.\,G.\ visited Stanford University in
April 2014, and he is very grateful to the Department of Mathematics for its 
hospitality and financial support.
J.\,P.\ subsequently visited the École normale supérieure de Lyon in July 
2014, and he thanks the Unité de mathématiques pures et appliquées for its 
hospitality and financial support.  This work was supported by the LABEX 
MILYON (ANR--10--LABX--0070) of Université de Lyon, within the program 
``Investissements d'Avenir'' (ANR--11--IDEX--0007) operated by the French 
National Research Agency (ANR).  J.\,P.\ was partially supported by an NSF 
Graduate Research Fellowship under grant number DGE--1147470.

\section{Review of complex geometry} \label{complexgeosec}

We now provide for the reader a review of some classical results in complex 
geometry which we need.  Our specific target is the solution of the 
$d''$-operator on Stein manifolds via the $L^2$ methods of H\"ormander 
\cite{hormander} and Andreotti--Vesentini \cite{andreottivesentini}.  This will 
be used later to construct the localized ``peak sections'' necessary for 
Donaldson's construction.  The reader may refer to 
\cite[Proposition 34]{donaldsonI} for an analogous discussion in the case of 
compact Kähler manifolds.

\subsection{Kähler geometry}

For a complex vector bundle $E$ with connection $d$ over a complex manifold $M$,
we denote by $d': E \otimes \Omega^{p,q} \to E \otimes \Omega^{p+1,q}$ and $d'':
E \otimes \Omega^{p,q} \to E \otimes \Omega^{p,q+1}$ the complex linear and 
complex conjugate linear parts of the exterior derivative $d: E \otimes
\Omega^k \to E \otimes \Omega^{k+1}$. When $M$ is equipped with a Kähler metric 
and $E$ is equipped with a Hermitian metric, we let $d'^*$ and $d''^*$
denote the formal adjoints of $d'$ and $d''$ respectively, and we let $\Delta'
:= d'^* d' + d' d'^*$ and $\Delta'' := d''^* d'' + d'' d''^*$ denote
the corresponding Laplacians.

Recall that on any holomorphic vector bundle with a Hermitian metric, there
exists a unique connection compatible with the metric and the holomorphic
structure, called the Chern connection.

\begin{lemma}[Bochner--Kodaira--Nakano identity]
Let $E$ be a holomorphic Hermitian vector bundle over a Kähler manifold. Then we
have
\begin{equation} \label{BKN}
   \Delta''_E = \Delta'_E + [i\Theta(E), \Lambda]
\end{equation}
where $\Theta(E)$ is the curvature of $E$ and $\Lambda$ is the adjoint of $L :=
\cdot \wedge \omega$.
\end{lemma}

For a holomorphic Hermitian vector bundle $E$ over a Kähler manifold, there is
an induced Hermitian metric on $E \otimes \Omega^{0,q}$.  The operator $d': E
\otimes \Omega^{0,q} \to E \otimes \Omega^{1,q} = E \otimes \Omega^{0,q} \otimes
\Omega^{1,0}$ further equips $E \otimes \Omega^{0,q}$ with an anti-holomorphic
structure. Together these induce a Chern connection on $E \otimes \Omega^{0,q}$.
We denote this connection by $\nabla = \nabla' + \nabla''$, where $\nabla' = 
d'$, and we denote the corresponding Laplacians by $\Box'$ and $\Box''$, where 
$\Box' = \Delta'$. Applying \eqref{BKN} to $E \otimes \Omega^{0,q}$ gives
\begin{equation} \label{BKNspecial}
   \Box''_{E \otimes \Omega^{0,q}} = \Box'_{E \otimes \Omega^{0,q}}
 + [i\Theta (E \otimes \Omega^{0,q}), \Lambda] .
\end{equation}
Now since $\Box'_{E \otimes \Omega^{0,q}} = \Delta'_E$ operating on $E \otimes
\Omega^{0,q}$, we may combine \eqref{BKN} and \eqref{BKNspecial} to produce the 
following Weitzenböck formula
\begin{equation} \label{weitzenbock}
   \Delta_E'' = \Box_{E \otimes \Omega^{0,q}}''
 + \Lambda i\Theta(E \otimes \Omega^{0,q}) - \Lambda i\Theta(E)
\end{equation}
operating on $E \otimes \Omega^{0,q}$. We remark, for clarity, that the first
composition is of maps $E \otimes \Omega^{0,q} \rightleftarrows E \otimes
\Omega^{0,q} \otimes \Omega^{1,1}$ and the second composition is of maps $E
\otimes \Omega^{0,q} \rightleftarrows E \otimes \Omega^{1,q+1}$. We have
followed Donaldson \cite[p36]{donaldsonYM} in the derivation of this identity.

\begin{lemma}[Morrey--Kohn--Hörmander formula]
Let $E$ be a holomorphic Hermitian vector bundle over a Kähler manifold $M$. For
any $u \in C^\infty_c (M, E \otimes \Omega^{0,q})$, we have
\begin{equation} \label{mkhidentity}
   \int |d''u|^2 + |d''^* u|^2 = \int |\nabla''u|^2
 + \int \langle u, \Lambda i\Theta (E \otimes \Omega^{0,q}) u \rangle
 - \langle u, \Lambda i\Theta(E) u \rangle .
\end{equation}
\end{lemma}

\begin{proof}
By the definition of the adjoint, integrating by parts gives
\begin{equation}
   \int |d''u|^2 + |d''^* u|^2
 = \int \langle u, \Delta''u \rangle .
\end{equation}
The same integration by parts with $\nabla$ in place of $d$ gives
\begin{equation}
   \int |\nabla''u|^2
 = \int \langle u, \Box''u \rangle .
\end{equation}
Now we take the difference of these two identities and use \eqref{weitzenbock}
to obtain \eqref{mkhidentity}.
\end{proof}

\subsection{\texorpdfstring{$L^2$}{L\textasciicircum 2} theory of the \texorpdfstring{$d''$}{d''}-operator}

The $L^2$ theory that we review here is due to Hörmander \cite{hormander} and
Andreotti--Vesentini \cite{andreottivesentini}.

\begin{lemma} \label{dapproxsmooth}
Let $E$ be a holomorphic Hermitian vector bundle over a complete Kähler manifold
$M$. We consider sections $u$ of $E \otimes \Omega^{p,q}$.
\begin{itemize}
\item
If $u, d''u \in L^2$ (in the sense of distributions), then there exists a
sequence $u_i \in C^\infty_c$ such that $(u_i, d''u_i) \to (u, d''u)$ in $L^2$.
\item
If $u, d''u, d''^* u \in L^2$ (in the sense of distributions), then there
exists a sequence $u_i \in C^\infty_c$ such that $(u_i, d''u_i, d''^* u_i)
\to (u, d''u, d''^* u)$ in $L^2$.
\end{itemize}
\end{lemma}

\begin{proof}
This is essentially a special case of Friedrichs' result \cite{friedrichs} which
applies more generally to any first order differential operator. We outline the 
argument, which is also given in Hörmander \cite[Proposition 2.1.1]{hormander}
and Andreotti--Vesentini \cite[Lemma 4, Proposition 5]{andreottivesentini}.

We prove the first statement only, as the proof of the second is identical. Let 
$u$ be given. Composing the distance function to a specified point in $M$ with
the cutoff function $x \mapsto \max (1 - \eps x, 0)$, we get a function $f_\eps:
M \to \R$ with $\sup |f_\eps| \le 1$ and $\sup |df_\eps| \le \eps$, 
so that $f_\eps \to 1$ uniformly on compact subsets of $M$ as $\eps \to 0$.
Using these properties, it follows that $f_\eps u \to u$ in $L^2$ and that $d''
(f_\eps u) \to d''u$ in $L^2$. Since $M$ is complete, $f_\eps$ is compactly
supported. Hence we may assume without loss of generality that $u$ is compact
supported.

Since $u$ is compactly supported, we may use a partition of unity argument to
reduce to the case when $u$ is supported in a given small coordinate chart of
$M$. Now in a small coordinate chart, choosing trivializations of the bundles in
question, the operator $d''$ is a first order differential operator $D$ with
smooth coefficients. It can now be checked (and this is the key point) that
$\|D (u* \varphi_\eps) - Du* \varphi_\eps\|_2 \to 0$, where $\varphi_\eps
:= \eps^{-n} \varphi(x/\eps)$ is a smooth compactly supported approximation to
the identity. It follows that the convolutions $u* \varphi_\eps$ give the
desired approximation of $u$ by smooth functions of compact support.
\end{proof}

\begin{proposition} \label{dbarexistence}
Let $E$ be a holomorphic Hermitian vector bundle over a complete Kähler manifold
$M$. Fix $q$, and suppose that for all $u \in C^\infty_c (M, E \otimes 
\Omega^{0,q})$, we have
\begin{equation}\label{coerciveestimate}
   \int |u|^2 \le A \int |d''u|^2 + |d''^* u|^2 .
\end{equation}
Then for any $u \in L^2 (M, E \otimes \Omega^{0,q})$ satisfying $d''u=0$, there 
exists $\xi \in L^2 (M, E \otimes \Omega^{0,q-1})$ satisfying $d''\xi = u$ and
\begin{equation}
   \int |\xi|^2 \le A \int |u|^2
\end{equation}
($d''$ is taken in the sense of distributions).
\end{proposition}

\begin{proof}
We follow an argument from notes by Demailly \cite[p33, (8.4) Theorem]{demailly}.

We wish to find $\xi$ such that $d''\xi = u$, or, equivalently, 
$\int \langle d''^*\varphi, \xi \rangle = \int \langle \varphi, u \rangle$ for 
all $\varphi \in C^\infty_c (M, E \otimes \Omega^{0,q})$. We claim that the
existence of such a $\xi$ with $\int |\xi|^2 \leq B$ is equivalent to the 
estimate
\begin{equation} \label{keyestimate}
   \left | \int \langle \varphi, u \rangle \right |^2 \leq B 
   \int |d''^*\varphi|^2
\end{equation}
for all $\varphi \in C^\infty_c (M, E \otimes \Omega^{0,q})$. Indeed, given
\eqref{keyestimate}, the map $d''^*\varphi \mapsto \int \langle\varphi,u\rangle$
on $d''^* (C^\infty_c (M, E \otimes \Omega^{0,q}))$ is well defined and $L^2$ 
bounded, and thus it is of the form $\int \langle d''^*\varphi, \xi \rangle$ for
a unique $\xi$ in the closure of $d''^* (C^\infty_c (M,E\otimes\Omega^{0,q}))
\subseteq L^2 (M, E \otimes \Omega^{0,q-1})$ satisfying $\int |\xi|^2 \leq B$. 
Thus we are reduced to showing \eqref{keyestimate} for $B = A \int|u|^2$.

To prove \eqref{keyestimate}, argue as follows. Since $L^2$ convergence implies 
distributional convergence, the kernel (in the sense of distributions) $\ker d''
\subseteq L^2 (M, E \otimes \Omega^{0,q})$ is a closed subspace. Hence for any
$\varphi \in C^\infty_c (M, E \otimes \Omega^{0,q})$, we may write $\varphi =
\varphi_1 + \varphi_2$ where $\varphi_1 \in \ker d''$ and $\varphi_2 \in
(\ker d'')^\perp$. Now since $u \in \ker d''$, we have
\begin{equation}\label{firstineq}
   \left | \int \langle \varphi, u \rangle \right |^2
 = \left | \int \langle \varphi_1, u \rangle \right |^2 \leq
   \int |u|^2 \cdot \int|\varphi_1|^2 .
\end{equation}
Since $\varphi_2 \perp \ker d'' \supseteq \im d''$, it follows that $\varphi_2
\in \ker d''^*$ (in the sense of distributions). Hence
\begin{equation} \label{secondineq}
   \int |d''\varphi_1|^2 + |d''^*\varphi_1|^2 = \int |d''^*\varphi|^2 .
\end{equation}
Combining \eqref{firstineq} and \eqref{secondineq}, we see that to prove \eqref
{keyestimate} with $B = A \int |u|^2$, it suffices to show that
\begin{equation}
   \int |\varphi_1|^2 \leq A \int |d''\varphi_1|^2 + |d''^*\varphi_1|^2 .
\end{equation}
This is true by hypothesis \eqref{coerciveestimate} for $\varphi_1 \in 
C^\infty_c (M, E \otimes \Omega^{0,q})$, and hence by Lemma \ref{dapproxsmooth} 
it holds given just that $\varphi_1, d''\varphi_1,  d''^*\varphi_1 \in L^2$.
\end{proof}

\subsection{Stein manifolds and solving the \texorpdfstring{$d''$}{d''}-operator}

Let $V$ be a Stein manifold or a Stein domain. A smooth $J$-convex function
$\phi: V \to \R$ induces a symplectic form $\omega_\phi := id'd''\phi$ and a
Riemannian metric $g_\phi(X,Y) := \omega_\phi(X,JY)$ (so $h_\phi := g_\phi -
i\omega_\phi$ is a Hermitian metric) whose distance function we denote by 
$d_\phi (\cdot,\cdot)$. The function $\phi$ also gives rise to a holomorphic
Hermitian line bundle $L^\phi$ over $V$, namely the trivial complex line
bundle $\C$ equipped with its standard holomorphic structure $d_\C''$ and the 
Hermitian metric $\left|\cdot\right|_{L^\phi} := e^{-\frac 12\phi} 
\left|\cdot\right|_\C$.  The resulting Chern connection on $L^\phi$ is given by
\begin{equation}
   d_{L^\phi} = d_\C - d'\phi
\end{equation}
with curvature $\Theta(L^\phi) = d'd''\phi = -i\omega_\phi$.  (Equivalently, 
$L^\phi$ is the trivial complex line bundle equipped with its standard 
Hermitian metric and the holomorphic structure $d''_\C + \frac 12 d''\phi$, 
with resulting Chern connection $d_\C + \frac 12 i J^* d\phi$.  This is 
equivalent to the first definition via multiplication by $e^{\frac 12\phi}$.)

The following result (due to Hörmander \cite{hormander} and Andreotti--Vesentini
\cite{andreottivesentini}) allows us to produce many holomorphic sections of 
$L^\phi$ for sufficiently $J$-convex $\phi$.

\begin{proposition} \label{steinsolved}
For every Stein manifold $V$ with complete Kähler metric $g$, there exists a
continuous function $c: V \to \R_{>0}$ with the following property. Let $\phi: V
\to \R$ be $J$-convex and satisfy $g_\phi \ge c \cdot g$ (pointwise inequality 
of quadratic forms). Then for any $u \in L^2 (V, L^\phi \otimes \Omega^{0,q})$ 
($q>0$) satisfying $d''u = 0$, there exists $\xi \in L^2 (V, L^\phi)$ satisfying 
$d''\xi = u$ and
\begin{equation} \label{dbarexistencesteinbound}
   \int |\xi|^2 \le \int |u|^2 .
\end{equation}
\end{proposition}

\begin{proof}
By Proposition \ref{dbarexistence}, it suffices to show the estimate
\begin{equation}
   \int |d''u|^2 + |d''^* u|^2 \geq \int |u|^2
\end{equation}
for all $u \in C^\infty_c (V, L \otimes \Omega^{0,q})$. Applying the
Morrey--Kohn--Hörmander identity \eqref{mkhidentity} to the left hand side, it
suffices to show the following pointwise curvature estimate
\begin{equation}
   \langle u, \Lambda i\Theta (L \otimes \Omega^{0,q}) u \rangle
 - \langle u, \Lambda i\Theta(L) u \rangle \geq |u|^2 .
\end{equation}
Expanding $\Theta(L \otimes \Omega^{0,q}) = \Theta(L) \otimes \id_{\Omega^{0,q}}
+ {\id_L} \otimes \Theta(\Omega^{0,q})$, it suffices to show that
\begin{equation} \label{secondcurvest}
   \langle u, \Lambda i(\Theta(L) \otimes \id) u \rangle
 - \langle u, \Lambda i\Theta(L) u \rangle \geq 
   (1 + |\Lambda| \, |\Theta (\Omega^{0,q})|) |u|^2 .
\end{equation}
We remark for clarity that the first composition is of maps $L \otimes
\Omega^{0,q} \rightleftarrows L \otimes \Omega^{0,q} \otimes \Omega^{1,1}$ and
the second composition is of maps $L \otimes \Omega^{0,q} \rightleftarrows
L \otimes \Omega^{1,q+1}$. Let $\alpha_1, \ldots, \alpha_n$ denote the scaling
factors associated to a simultaneous diagonalization of $g$ and $g_\phi$,
meaning that $|v_i|^2_{g_\phi} = \alpha_i |v_i|^2_g$ for a simultaneous
orthogonal basis $v_1, \ldots, v_n$. We may now calculate (see Voisin \cite
[Lemma 6.19]{voisin})
\begin{equation}
   \Lambda i(\Theta(L) \otimes \id) u = \biggl(\sum_{i=1}^n \alpha_i\biggr) u .
\end{equation}
The operator $\Lambda i\Theta(L)$ has an orthonormal basis of eigenvectors with
eigenvalues $\sum_{i \in I} \alpha_i$ for all $I \subseteq \{1, \ldots, n\}$ 
with $|I| = n-q$. Thus to ensure \eqref{secondcurvest}, it suffices to have
$q \min \alpha_i \geq 1 + |\Lambda| \, |\Theta(\Omega^{0,q})|$, which can be
achieved by choosing $c = 1 + |\Lambda| \, |\Theta(\Omega^{0,q})|$ since $q>0$.
\end{proof}

\begin{lemma} \label{localJconvex}
Let $\phi: (\C^n, 0) \to \R$ be a germ of a smooth $J$-convex function. For all
$\eps>0$, there exists a germ of a holomorphic function $u: (\C^n,0) \to \C$
satisfying
\begin{equation}
   \left|\Re u(z) - \Bigl[\phi(z) - \frac12 d_\phi(z,0)^2\Bigr]\right| \le
   \eps \cdot d_\phi (z,0)^2
\end{equation}
in a neighborhood of zero.
\end{lemma}

\begin{proof}
The statement depends only on $\phi$ up to second order, so we may assume without loss of generality that $\phi$ is a real degree two polynomial on $\C^n$.  Any real polynomial on $\C^n$ may be expressed uniquely as a polynomial in $z_i$ and $\bar z_i$ with coefficients $c_{i_1,\ldots,i_k,\bar i_1,\ldots,\bar i_\ell}\in\C$ satisfying $c_{i_1,\ldots,i_k,\bar i_1,\ldots,\bar i_\ell}=\ol{c_{i_1,\ldots,i_\ell,\bar i_1,\ldots,\bar i_k,}}$.  In the case of degree two, we thus have
\begin{equation}
   \phi(z) = a + \sum_i \Re a_i z_i + \sum_{i,j} \Re a_{ij} z_i z_j
 + \sum_{i,j} b_{ij} z_i \bar z_j
\end{equation}
where $a \in \R$, $a_i, a_{ij}, b_{ij} \in \C$, and $b_{ij} = \ol{b_{ji}}$. 
The statement is also unaffected by adding the real part of a holomorphic 
function to $\phi$, so we may assume that $a = a_i = a_{ij} = 0$. Finally, the 
statement is unaffected by precomposing $\phi$ with a germ of biholomorphism of 
$\C^n$ near zero, so we may apply an element of $\operatorname{GL}_n(\C)$ so 
that the positive definite Hermitian matrix $(b_{ij})$ becomes the identity 
matrix. Hence we have without loss of generality that $\phi(z) = |z|^2$, for 
which we may take $u \equiv 0$.
\end{proof}

\section{Donaldson's construction} \label{globalconstruction}

We now prove Theorem \ref{quantitativesteinexistence}.

Let us begin by fixing some notation/terminology. We fix a Stein manifold $\ol
V$ and a smooth exhausting $J$-convex function $\phi: \ol V \to \R$. We let $V
:= \{\phi \le 0\}$, so $\partial V = \{\phi = 0\}$. We denote by $g := g_\phi$ 
the induced metric on $\ol V$, with associated distance function $d := d_\phi$. 
We denote by $L := L^\phi$ the associated line bundle.  For any positive real
number $k$, we let $g_k := g_{k\phi} = kg$, $d_k := d_{k\phi} = k^{1/2}d$, and 
$L^k := L^{k\phi}$.

In what follows, we treat $k$ as a fixed real parameter, and most statements 
(in particular, the notations $O(\cdot)$ and $o(\cdot)$) are meant in the limit 
$k \to \infty$ (i.e.\ for $k$ sufficiently large).  Most implied constants are 
independent of $(\ol V, \phi)$ (unless stated otherwise), however how large 
$k$ must be may (and almost always will) depend on $(\ol V, \phi)$.

Near any point $p_0 \in \ol V$, there exists a holomorphic coordinate chart
$\Psi: (U, 0) \to (\ol V, p_0)$, where $U \subseteq \C^n$ is an open subset
containing zero, and a holomorphic function $u: \Psi(U) \to \C$, satisfying:
\begin{itemize}
\item
$B(r) \subseteq U$ for $r^{-1} = O(1)$.
\item
$\Psi^* \phi = a\Re z_1 + O(|z|^2)$ if $p_0 \in \partial V$, 
where $a=\left|d\phi(p_0)\right|$.
\item
$\Psi^* g = g_{\C^n} + O(|z|)$.
\item
$\phi(p) - \frac 34 d(p,p_0)^2 \le \Re u(p) \le \phi(p) - \frac 14 d(p,p_0)^2$.
\end{itemize}
(for the existence of $u$, we appeal to Lemma \ref{localJconvex}). There exists 
such a triple $(U, \Psi, u)$ for which the implied constants above are bounded
as $p_0$ varies over any compact subset of $\ol V$. It is convenient to also
have at our disposal the rescaled coordinates $\Psi_k: (B(2),0) \to (\ol V,p_0)$
defined by $\Psi_k(\cdot) = \Psi(k^{-1/2}\cdot)$ and the rescaled function $ku$ 
(for sufficiently large $k$), which satisfy:
\begin{itemize}
\item
$\Psi_k^* \phi = ak^{-1/2} \Re z_1 + O(k^{-1} |z|^2)$ if $p_0 \in \partial V$, 
where $a = \left| d\phi(p_0) \right|$.
\item
$\Psi_k^* g_k = g_{\C^n} + O(k^{-1/2} |z|)$.
\item
$k\phi(p) - \frac34 d_k(p,p_0)^2 \le \Re ku(p) \le k\phi(p) - \frac 14
d_k(p,p_0)^2$.
\end{itemize}
Now the section $\sigma := e^{\frac12 ku}$ of $L^k$ satisfies
\begin{equation} \label{sigmaisbounded}
   e^{-\frac38 d_k(p,p_0)^2} \le |\sigma(p)| \le
   e^{-\frac18 d_k(p,p_0)^2}
\end{equation}
over its domain of definition $\Psi(U)$. This ``reference section'' provides a
convenient local holomorphic trivialization of $L^k$ over $\Psi_k (B(2))$. We 
also need holomorphic sections of $L^k$ defined on all of $\ol V$ which satisfy 
a decay bound similar to \eqref{sigmaisbounded} over $\{\phi \leq 1\}$ and 
which approximate $\sigma$ over $\Psi_k(B(2))$. That such sections exist is 
the content of the following lemma.

\begin{lemma} \label{peaksections}
Let $(\ol V, \phi)$ be as above. Fix $p_0 \in \{\phi = 0\}$ and consider the 
associated coordinates $\Psi$ and reference section $\sigma$ as above. There 
are holomorphic sections $\tilde\sigma, \tilde\sigma_1, \ldots, \tilde\sigma_n: 
\ol V \to L^k$ satisfying:
\begin{itemize}
\item
$|\tilde\sigma(p)| \leq e^{-\frac19 d_k(p,p_0)^2} + e^{-\eps k}$ over $\{\phi 
\le 1\}$.
\item
$|\tilde\sigma_r(p)| \leq e^{-\frac19 d_k(p,p_0)^2} + e^{-\eps k}$ over $\{\phi 
\le 1\}$ for $r = 1, \ldots, n$.
\item
$\left|\frac{ \tilde\sigma } \sigma \circ \Psi_k - 1\right| \leq e^{-\eps k}$ 
over $B(2)$.
\item
$\left|\frac{ \tilde\sigma_r } \sigma \circ \Psi_k - z_r\right| \leq 
e^{-\eps k}$ over $B(2)$ for $r = 1, \ldots, n$.
\end{itemize}
for some $\eps>0$ depending on $(\ol V,\phi)$ and sufficiently large $k$.
\end{lemma}

\begin{proof}
Fix a smooth cutoff function $\beta: \ol V \to [0,1]$ supported inside $\Psi(U)$
which equals $1$ in a neighborhood of $p_0$. Now $\|d'' (\beta  \sigma)\|_2 \le
e^{-\eps k}$ in the fixed metric $g$ for sufficiently large $k$ and some $\eps
> 0$ depending on $(\ol V, \phi)$.

Fix a smooth exhausting $J$-convex function $\phi_1: \ol V \to \R$ which
coincides with $\phi$ over $\{\phi \le 2\}$ and for which $g_{k\phi_1} \ge c
\cdot g$ for sufficiently large $k$ (for $c$ as in Proposition \ref
{steinsolved}). We apply Proposition \ref{steinsolved} to $(\ol V, g, k\phi_1)$ 
and conclude that there exists a section $\xi$ of $L^k$ for which $\beta \sigma 
+ \xi$ is holomorphic and $\|e^{\frac12  k \cdot (\phi-\phi_1)}\xi\|_2 \le
\|d'' (\beta \sigma)\|_2 \leq e^{-\eps k}$.

Let us now show that $\tilde\sigma := \beta \sigma + \xi$ satisfies the desired 
properties.  Over the set where $\beta=1$, the section $\xi$ is holomorphic. In
particular, the function $\frac \xi \sigma \circ \Psi_k$ is holomorphic over 
$B(3)$ (for sufficiently large $k$).  We have $\|\frac \xi \sigma \circ \Psi_k\|
_{ B(3), 2 } \leq e^{-\eps k}$, from which it follows that $|\frac \xi \sigma
\circ \Psi_k| \leq e^{-\eps k}$ over $B(2)$ (for a possibly smaller $\eps>0$ and
larger $k$) since $\frac \xi \sigma \circ \Psi_k$ is holomorphic.  Thus we have 
$\left|\frac{ \tilde\sigma } \sigma \circ \Psi_k - 1\right| \leq e^{-\eps k}$
over $B(2)$.

Now let $p \in \{\phi \leq 1\}$ and consider the associated coordinates $\Psi'$ 
and reference section $\sigma'$ as above. We have $\|\frac{ \tilde\sigma }
{ \sigma' } \circ \Psi_k'\|_{ B(3),2 } = O(e^{-\frac18 d_k(p,p_0)^2} +
e^{-\eps k})$, from which it follows that $|\tilde\sigma(p)| = O (e^{-\frac18
d_k(p,p_0)^2} + e^{-\eps k})$ (since $\frac{ \tilde\sigma }{ \sigma' } \circ
\Psi_k'$ is holomorphic), which gives the desired decay bound on $\tilde\sigma$.

The argument for $\{\tilde\sigma_r\}_{1 \leq r \leq n}$ is identical, with 
$(z_r \circ \Psi_k^{-1}) \cdot \sigma$ in place of $\sigma$.
\end{proof}

It is helpful to rephrase Theorem \ref{quantitativesteinexistence} as follows in
terms of the line bundle $L^k$ and the rescaled metric $g_k$ on $\ol V$.

\begin{theorem} \label{quantitativesteinexistencerescaled}
Let $\ol V$ be a Stein manifold, equipped with a smooth exhausting $J$-convex
function $\phi: \ol V \to \R$. For every sufficiently large real number $k$,
there exists a holomorphic section $s: \ol V \to L^k$ such that:
\begin{itemize}
\item
$|s(p)| \le 1$ for $p \in \{\phi \le 1\}$.
\item
$|s(p)| + |ds(p) \rst \xi| > \eta$ for $p \in \{\phi = 0\}$ ($ds$ measured in 
the metric induced by $k\phi$).
\end{itemize}
where $\xi$ denotes the Levi distribution on $\{\phi = 0\} \subseteq \ol V$, and 
$\eta>0$ is a constant depending only on the dimension of $\ol V$.
\end{theorem}

\begin{proof}
The proof follows Donaldson \cite[\S 3]{donaldsonI}, as simplified by Auroux
\cite{aurouxremark}.

\textbf{Part I.}
Fix a maximal collection of points $p_1, \ldots, p_N \in \partial V$ whose 
pairwise $d_k$-distances are $\geq 1$. Since this collection is maximal, the
unit $d_k$-balls $B_i$ centered at the $p_i$'s cover $\partial V$. The
$d_k$-balls of radius $1/2$ centered at the $p_i$'s are disjoint, so by volume
considerations, the total number of points satisfies $N = O_{(\ol V,\phi)}
(k^{2n-1})$, where $n$ is the complex dimension of $\ol V$.

We now specify the form of the section $s: \ol V \to L^k$ we will construct.
For each $p_i$, we will define a holomorphic section $s_i: \ol V \to L^k$
satisfying the bound
\begin{equation} \label{boundonsi}
   |s_i(p)| \leq e^{-\frac19 d_k(p,p_i)^2} + e^{-\eps k} \quad \text{for }
   p \in \{\phi \le 1\}
\end{equation}
(for some $\eps>0$ depending on $(\ol V, \phi)$) and we will let
\begin{equation}
   s := \sum_{i=1}^N s_i .
\end{equation}
Let us observe immediately that this bound on $|s_i|$ implies that
\begin{equation*}
   |s(p)| \leq \sum_{i=1}^N e^{-\frac19 d_k(p,p_i)^2} + e^{-\eps k} \approx
   \int_{\partial V} e^{-\frac19 d_k(p,p_0)^2} dg_k(p_0)
 + O_{(\ol V,\phi)} (k^{2n-1} e^{-\eps k}) = O(1)
\end{equation*}
for $p \in \{\phi \le 1\}$. In particular, this ensures the first condition
$|s(p)| \le 1$ (after dividing by a constant factor depending only on $n = \dim
\ol V$).

\begin{remark}[$C^0$-bounds imply $C^\infty$-bounds for holomorphic functions]
For a holomorphic function $f$ defined on $B(1+\epsilon) \subseteq \C^n$, we 
have
\begin{equation}\label{ellipticestimate}
   \|f\|_{C^\ell(B(1))} \leq c_{n,\ell}
   \Bigl(1 + \frac 1 {\epsilon^\ell}\Bigr) \|f\|_{C^0(B(1+\epsilon))}.
\end{equation}
(Indeed, we have $|D^\ell f(0)| \leq c_{n,\ell} \sup_{B(1)}f$ by the Cauchy
integral formula, and applying this to balls of radius $\epsilon>0$ along with
the maximum principle yields the above estimate).

For simplicity of notation, we have stated the upper bounds in
\eqref{sigmaisbounded}, Lemma \ref{peaksections}, \eqref{boundonsi}, and
\eqref{etaboundforsi} below only in the $C^0$-norm, though of course we will
often need to use the resulting bounds on higher derivatives implied by
\eqref{ellipticestimate}.  If we were working in the approximately holomorphic
setting, we would need to explicitly bound the higher derivatives up to some 
appropriate fixed finite order.
\end{remark}

\begin{definition}
A section $s: V \to L^k$ will be called \emph{$\eta$-transverse at $p \in 
\partial V$} iff $|s(p)| + |ds(p) \rst \xi| > \eta$.  The property of being 
$\eta$-transverse is obviously stable under $C^1$-perturbation, and for 
holomorphic sections it is in fact stable under $C^0$-perturbation by 
\eqref{ellipticestimate} with $\ell=1$ as long as the perturbation is defined in a 
fixed neighborhood of $p$.
\end{definition}

\begin{remark}
This particular quantitative transversality condition was first considered by 
Mohsen \cite{mohsenpreprint}, and is closely related to those used by Donaldson 
and Auroux.  Donaldson \cite{donaldsonI} called a section $s: V \to L^k$ 
$\eta$-transverse at $p$ iff either $|s(p)| \ge \eta$ or $|ds(p)| \ge \eta$ 
(this is equivalent, up to a constant, to requring $|s(p)| + |ds(p)| \ge \eta$).
Mohsen \cite{mohsenpreprint} generalized this notion to quantitative 
transversality relative to a given submanifold $Y$.  Specifically, he called a 
section $\eta$-transverse relative to $Y$ iff either $|s(p)| \ge \eta$ or $ds(p)
\rst {TY}$ has a right inverse of norm $\le \eta^{-1}$.  In the case of the 
submanifold $\partial V \subseteq \ol V$ and an (approximately) holomorphic 
section $s$, this condition is equivalent, up to a constant, to our formulation 
$|s(p)| + |ds(p) \rst \xi| > \eta$ (see \cite[\S2]{mohsenpreprint}).  Thus, 
Theorem \ref{quantitativesteinexistencerescaled} can be regarded as a 
holomorphic version of Mohsen's transversality theorem for hypersurfaces.
\end{remark}

\textbf{Part II.}
Our goal is to construct sections $s_i$ satisfying the decay bound 
\eqref{boundonsi} so that $s$ is $\eta$-transverse over $\partial V$ for some 
$\eta>0$ depending only on $n$.

We will define the sections $s_i$ in a series of steps, at each step achieving
(quantitative) transversality over some new part of $\partial V$, while
maintaining (quantitative) transversality over the part of $\partial V$ already
dealt with. The most naive version of this procedure, choosing $s_i$ to achieve 
transversality over $B_i$ while maintaining transversality over $B_1, \ldots,
B_{i-1}$, runs into trouble, essentially due to the rather large number of
steps. Instead, we first construct a suitable coloring of the $p_i$'s, and then 
in the inductive procedure we choose the $s_i$'s for the $p_i$'s of a particular
color simultaneously (so there is one step per color). For this to work, we must
ensure that points of the same color are sufficiently far apart.

Let $D < \infty$ be a (large) positive real number, to be fixed (depending only 
on $n$) at the end of the proof. We color the $p_i$'s so that the $d_k$-distance
between any pair of points of the same color is at least $D$. More precisely, we
construct such a coloring by iteratively choosing a maximal collection of yet
uncolored points $p_i$ with pairwise distances $\geq D$ and then coloring this
collection with a new color. Because each color was chosen from a maximal
collection of yet uncolored points, it follows that the ball of radius $D$
centered at any point colored with the final color contains points of every
other color. Hence by volume considerations, it follows that the total number of
colors $M$ is $O(D^{2n-1})$. Let us denote the coloring function by $c: \{1,
\ldots, N\} \to \{1, \ldots, M\}$.

\textbf{Part III.}
Let $\p < \infty$ and $A < \infty$ be (large) positive real numbers, to be fixed 
(depending only on $n$) later in the proof. To be precise, we must first choose 
$A$ (depending on $n$), then choose $\p$ (depending on $n$ and $A$), and finally 
choose $D$ (depending on $n$, $A$, and $\p$).

It suffices to construct sections $s_i$ so that:
\begin{itemize}
\item
For all $j \in \{1, \ldots, M\}$ and $c(i) = j$, we have
\begin{equation} \label{etaboundforsi}
   |s_i(p)| \leq \frac1A \eta_{j-1}
    \left[e^{-\frac19 d_k(p,p_i)^2} + e^{-\eps k}\right] \quad 
   \text{for }p \in \{\phi \le 1\} .
\end{equation}
\item
For all $j \in \{1, \ldots, M\}$, we have
\begin{equation}
   s^j := \sum_{\begin{smallmatrix} i=1 \cr c(i) \leq j\end{smallmatrix}}^N s_i
   \quad \text{is $\eta_j$-transverse over} \quad
   X_j := \bigcup_{\begin{smallmatrix} i=1 \cr c(i) \leq j \end{smallmatrix}}^N
   B_i .
\end{equation}
\end{itemize}
Here $\frac14 = \eta_0 > \eta_1 > \cdots > \eta_M > 0$ are defined by $\eta_j 
= \eta_{j-1} \left|\log\eta_{j-1}\right|^{-\p}$ (the reason for this particular 
choice will become apparent later).

We construct such sections $s_i$ by induction on $j$. More precisely, it
suffices to suppose that sections $s_i$ are given for $c(i) \le j-1$ (satisfying
the above in the range $\{1, \ldots, j-1\}$) and to construct sections $s_i$ for
$c(i) = j$ (satisfying the above in the range $\{1, \ldots, j\}$).

\textbf{Part IV.}
As a first step, let us fix an index $i$ with $c(i) = j$, and construct a
section $s_i$ satisfying \eqref{etaboundforsi} so that $s^{j-1} + s_i$ is
$\eta_{j-1} \left|\log\eta_{j-1}\right|^{-\p}$-transverse over $B_i$ (for some 
$\p < \infty$ depending on $n$ and $A$).

Fix a triple $(U, \Psi, u)$ based at $p_i \in \partial V$ (as discussed at the 
beginning of this section), with rescaling $\Psi_k$ and reference section 
$\sigma = e^{\frac12 ku}$. We will use the local coordinates $\Psi_k$ and the
reference section $\sigma$ to measure the transversality of $s^{j-1} + s_i$ over
$B_i$. Precisely, we claim that it suffices to construct $s_i$ satisfying \eqref
{etaboundforsi} so that
\begin{equation} \label{dividedsection}
   \frac{ s^{j-1} + s_i } \sigma \circ \Psi_k
\end{equation}
is $\eta_{j-1} \left|\log\eta_{j-1}\right|^{-\p}$-transverse over $B(3/2) \cap 
\Psi_k^{-1}(\partial V)$. Indeed, $\sigma$ is bounded above and below by \eqref
{sigmaisbounded}, so using \eqref{ellipticestimate} with $\ell=1$ this implies 
that $s^{j-1} + s_i$ is $\frac1C \eta_{j-1}
\left|\log\eta_{j-1}\right|^{-\p}$-transverse over $B_i$ for some constant 
$C < \infty$ depending only on $n$ (which we can absorb into the last factor 
by increasing $\p$).

Now as $k \to \infty$, the real hypersurface $B(3/2) \cap \Psi_k^{-1} (\partial
V)$ approaches $B(3/2) \cap \{\Re z_1 = 0\}$ in $C^\infty$, uniformly over the
choice of $p_i \in \partial V$. Since \eqref{dividedsection} is bounded
uniformly over $B(2)$, using \eqref{ellipticestimate} with $\ell=2$ we see that 
$\eta$-transversality over $B(3/2) \cap \{\Re z_1 = 0\}$
implies $(\eta-o(1))$-transversality over $B(3/2) \cap \Psi_k^{-1}(\partial V)$ 
(of course, the condition of $\eta$-transversality over a real hypersurface 
is with respect to its own Levi distribution).  
Since the number of colors $M$ is bounded independently of $k$, it follows that
$\eta_{j-1}$ is bounded away from zero as $k \to \infty$. Hence it suffices to
show that \eqref{dividedsection} is 
$\eta_{j-1}\left|\log\eta_{j-1}\right|^{-\p}$-transverse
over $B(3/2) \cap \{\Re z_1 = 0\}$ (we again lose a constant on the
transversality estimate, but as before it can be absorbed into the exponent
$\p$).

For any vector $w = (w_0, w_2, \ldots, w_n) \in \C^n$, we consider the 
holomorphic function on $B(2)$ given by
\begin{equation} \label{newtransverselocal}
   \frac{ s^{j-1} } \sigma \circ \Psi_k + w_0 + w_2z_2 + \cdots + w_nz_n .
\end{equation}
A quantitative transversality theorem, Proposition \ref{qtrans} (whose proof
we defer to later) says that for $\frac 13 > \eta > 0$, there exists a vector
$w = (w_0, w_2, \ldots, w_n) \in \C^n$ with $|w| \leq \eta$ so that \eqref
{newtransverselocal} is $\eta \left|\log\eta\right|^{-\p}$-transverse over 
$B(3/2) \cap \{\Re z_1 = 0\}$ (for some $\p < \infty$ depending only on $n$).
This fact that with a perturbation of size $\eta$ we can achieve 
$\eta \left|\log\eta\right|^{-\p}$-transversality is what forces the choice 
of recursion $\eta_j = \eta_{j-1} \left|\log\eta_{j-1}\right|^{-\p}$ 
declared above.

Let $\tilde\sigma$ and $\{\tilde\sigma_r\}_{1\leq r\leq n}$ denote the ``peak 
sections'' based at $p_0=p_i$ from Lemma \ref{peaksections}.  We define 
$s_i := w_0\tilde\sigma + w_2\tilde\sigma_2 + \ldots + w_n \tilde\sigma_n$ 
(for $w$ to be determined), so now \eqref{dividedsection} equals
\begin{equation} \label{newdividedsection}
   \frac{ s^{j-1} } \sigma \circ \Psi_k + w_0 \frac{\tilde\sigma}\sigma 
   \circ \Psi_k + w_2 \frac{\tilde\sigma_2}\sigma \circ \Psi_k + \cdots + 
   w_n \frac{\tilde\sigma_n}\sigma \circ \Psi_k .
\end{equation}
There is a constant $C < \infty$ (depending only on $n$) such that for 
$|w| \leq \frac 1 {A \cdot C}
\eta_{j-1}$, the section $s_i$ satisfies the decay bound \eqref{etaboundforsi}. 
By Proposition \ref{qtrans}, there exists $|w| \leq \frac 1 {A \cdot C}
\eta_{j-1}$ for which \eqref{newtransverselocal} is $\eta_{j-1}
\left|\log\eta_{j-1}\right|^{-\p}$-transverse over $B(3/2) \cap \{\Re z_1 = 0\}$ 
(absorbing constants into $\p$). It follows that \eqref{newdividedsection} (and 
hence \eqref{dividedsection}) is $(\eta_{j-1}
\left|\log\eta_{j-1}\right|^{-\p} - O(e^{-\eps k}))$-transverse over $B(3/2) 
\cap \{\Re z_1 = 0\}$, which is enough.

\textbf{Part V.}
We have constructed sections $s_i$ for $c(i) = j$ with the property that
$s^{j-1} + s_i$ is $\eta_{j-1} \left|\log\eta_{j-1}\right|^{-\p}$-transverse over
$B_i$ (for some $\p < \infty$ depending on $n$ and $A$). Now let us argue that
with this choice of sections, $s^j$ is $\eta_j$-transverse over $X_j$ (for some 
possibly different $\p < \infty$ depending on $n$ and $A$).

We know that $s^j$ differs from $s^{j-1}$ over $X_{j-1}$ by $O (\frac1A
\eta_{j-1})$ and that $s^{j-1}$ is $\eta_{j-1}$-transverse over $X_{j-1}$. It
follows that $s^j$ is $(1 - O (\frac1A)) \eta_{j-1}$-transverse over $X_{j-1}$, 
which gives $\eta_j$-transversality over $X_{j-1}$ once $A$ and $\p$ are large.

We know that $s^j$ differs from $s^{j-1} + s_i$ over $B_i$ by $O (\eta_{j-1}
e^{-\frac19 D^2})$ and that $s^{j-1} + s_i$ is $\eta_{j-1} 
\left|\log\eta_{j-1}\right|^{-\p}$-transverse over $B_i$. It follows that $s^j$ 
is $(\eta_{j-1} \left|\log\eta_{j-1}\right|^{-\p} - 
O (\eta_{j-1} e^{-\frac19 D^2}))$-transverse over $B_i$.
This gives $\eta_j$-transversality over $B_i$ (increasing $\p$ to make up for
the lost constant factor) as long as we have
\begin{equation} \label{keyinequality}
   e^{-\frac19 D^2} \le \frac1B \left|\log\eta_{j-1}\right|^{-\p}
\end{equation}
for some constant $B < \infty$ depending only on $n$.

Hence we conclude that the entire construction succeeds as long as \eqref
{keyinequality} holds for $j = 1, \ldots, M$. It is elementary to observe that 
the recursive definition of $\eta_j$ yields rough asymptotics $\eta_j \approx 
e^{-c \cdot j \log j}$ ($c$ depending on $\p$). Thus it suffices to ensure that
\begin{equation}
   e^{-\frac19 D^2} \le \frac1 {B'} (M \log M)^{-\p}
\end{equation}
for some $B' < \infty$ depending on $n$ and $\p$. We observed earlier that $M =
O(D^{2n-1})$, so this inequality is satisfied once $D$ is sufficiently large.
\end{proof}

\begin{remark}
A common theme in $h$-principle arguments à la Gromov, in which we want to
construct some structure globally on a given manifold $X$, is to extend the
desired structure to larger and larger subsets $\cdots \subseteq X_{j-1} 
\subseteq X_j \subseteq \cdots$ in a series of steps. This reduces the desired 
result to an extension problem from $X_{j-1}$ to $X_j$ (see for example 
Eliashberg--Mishachev \cite{hprinciple}). For example, $X_j$ is usually taken 
to be (an open neighborhood of) the $j$-skeleton of $X$ (under a fixed 
triangulation), the point being that now the \emph{topology} governing the 
extension from $X_{j-1}$ to $X_j$ is easy to understand. Donaldson's method, 
used in the proof above, employs a similar inductive procedure, but where one 
instead controls the \emph{geometry} governing the extension from $X_{j-1}$ 
to $X_j$ (the key point being that we can do local modifications 
\emph{independently} at any collection of points which are sufficiently far away
from each other).
\end{remark}

\section{Quantitative transversality theorem}\label{qtranssec}

We now prove the quantitative transversality theorem (Proposition \ref{qtrans}) 
which was the key technical ingredient in Donaldson's construction as used in 
the proof of Theorem \ref{quantitativesteinexistencerescaled}. The statement and
proof are similar to Auroux \cite[\S 2.3]{aurouxII}; see also 
\cite{aurouxremark}. A key ingredient is an upper bound on the volume of tubular
neighborhoods of real algebraic sets (Lemma \ref{wongkewtubularneighborhood}) 
due to Wongkew \cite{wongkew}.

\begin{proposition} \label{qtrans}
Let $B(1) \subseteq B(1+\eps) \subseteq \C^n$ be the balls centered at zero. 
Fix a holomorphic function $f: B(1+\eps) \to \C$ with $|f| \le 1$. For a vector
$w = (w_0, w_2, \ldots, w_n) \in \C^n$, we define
\begin{equation}
   f_w := f + w_0 + w_2 z_2 + \cdots + w_n z_n.
\end{equation}
For all $1/3 > \eta > 0$, there exists a vector $w \in \C^n$ satisfying $|w| \le
\eta \left|\log\eta\right|^\p$ such that:
\begin{itemize}
\item
$|f_w(z)| + |df_w(z) \rst \xi| > \eta$ for $z \in B(1)$ with $\Re z_1 = 0$.
\end{itemize}
where $\xi$ denotes the Levi distribution of $\{z \in B(1): \Re z_1 = 0\}$, and 
$\p < \infty$ depends only on the dimension $n$ and $\eps>0$.
\end{proposition}

\begin{remark}
A stronger version of Proposition \ref{qtrans} (a true quantitative Sard
theorem where we only perturb $f$ by a constant, i.e.\ $w_2 = \cdots = w_n=0$ 
above) is due to Donaldson \cite{donaldsonI, donaldsonII} and Mohsen
\cite{mohsenpreprint} with a rather more difficult proof.  Mohsen's result could
be used in \S\ref{globalconstruction} in place of Proposition \ref{qtrans},
resulting in a simpler definition $s_i := w_0 \tilde\sigma$, eliminating the
need for the remaining $\tilde\sigma_2, \ldots, \tilde\sigma_n$.  We have chosen
instead to present the argument following Auroux's observation that the weaker
Proposition \ref{qtrans}, whose proof is more elementary, is sufficient for the
argument in \S\ref{globalconstruction}.
\end{remark}

\begin{proof}
For a given $z \in B(1)$ with $\Re z_1 = 0$, the quantity $|f_w(z)| + |df_w(z)
\rst \xi|$ vanishes for exactly one value of $w$. The function $F: \{z \in B(1):
\Re z_1 = 0\} \to \C^n$ which associates to a given $z$ this unique $w$ is the 
restriction of a holomorphic function $F: B(1+\eps) \to \C^n$. Explicitly
\begin{equation}
   F(z) = \biggl(-f + z_2 \der_{z_2} f + \cdots + z_n \der_{z_n} f,
   -\der_{z_2} f, \ldots, -\der_{z_n} f\biggr) .
\end{equation}
In fact, the quantity $|f_w(z)| + |df_w(z) \rst \xi|$ is bounded below by (a
constant depending only on $n$, times) the distance from $w$ to $F(z)$. Hence it
suffices to show that $B(\delta) \setminus N_\eta (F (\{z \in B(1): \Re z_1 =
0\}))$ is nonempty for $\delta = \eta \left|\log\eta\right|^{O(1)}$.

We may approximate $F$ to within error $\le \eta$ on $B(1)$ by a polynomial
$\tilde F$ of degree $O(\left|\log\eta\right|)$.  Indeed, the error in the
degree $m$ Taylor approximation of $F$ is exponentially small in $m$, uniformly
over $B(1)$, since $F$ is holomorphic and bounded effectively on
$B(1+\frac\epsilon 2)$ by \eqref{ellipticestimate} with $\ell=1$.  To see this,
observe that (by the $U(n)$ symmetry) it is enough to prove an effective
exponential upper bound on the error over $B(1)\cap(\C\times\{0\}^{n-1})$, and
this is just the well-known single-variable case (proved using the Cauchy
integral formula).

It thus suffices to show that $B(\delta) \setminus N_{2\eta} 
(\tilde F (\{z \in B(1): \Re z_1 = 0\}))$ is nonempty for $\delta = \eta 
\left|\log\eta\right|^{O(1)}$.  Since $\tilde F$ is a polynomial of degree $O
(\left|\log\eta\right|)$, a pidgeonhole principle argument (Lemma 
\ref{imageinhypersurface} below) implies that its image is contained in a real 
algebraic hypersurface $X \subseteq \C^n$ of degree $\le 
\left|\log\eta\right|^{O(1)}$. Hence it suffices to show that $B(\delta) 
\setminus N_{2\eta}(X)$ is nonempty for $\delta = \eta 
\left|\log\eta\right|^{O(1)}$ and any real hypersurface 
$X \subseteq \C^n$ of degree $\le \left|\log\eta\right|^{O(1)}$.

Wongkew's estimate \cite{wongkew} (Lemma \ref{wongkewtubularneighborhood} below)
on the volume of a tubular neighborhood of a real algebraic variety gives
\begin{equation}
   \vol_{2n} (N_{2\eta} (X) \cap B(\delta))
 = \delta^{2n} \cdot O \left(\frac \eta \delta \left|\log\eta\right|^{O(1)}\right) .
\end{equation}
For $\delta = \eta\left|\log\eta\right|^{O(1)}$, this is less than the total 
volume of $B(\delta)$, which is enough.
\end{proof}

\begin{lemma}[{Auroux \cite[p565]{aurouxII}}] \label{imageinhypersurface}
Let $F: \R^n \to \R^m$ be a real polynomial map of degree $\le d$ where $n < m$.
Then the image of $F$ is contained in a real algebraic hypersurface of degree
$D \le \bigl\lceil (\frac{ m! }{ n! } d^n)^{ 1 / (m-n) } \bigr\rceil$.
\end{lemma}

\begin{proof}
The space of real polynomials $G$ of degree $\le D$ on $\R^m$ has dimension
$\binom{ m+D } m$. The composition $G \circ F$ has degree $\le dD$. Hence there
exists a nonzero $G$ for which the composition is zero provided $\binom{ m+D } m
> \binom{ n+dD } n$, or equivalently $\frac{ (D+1) \cdots (D+m) }{ (dD+1) \cdots
(dD+n) } > \frac{ m! }{ n! }$. The left hand side is bounded below by $\frac
{ D^m }{ (dD)^n }$, and so there exists a suitable $G$ as long as $D^{m-n} \ge
\frac{ m! }{ n! } d^n$.
\end{proof}

\begin{lemma}[Wongkew \cite{wongkew}] \label{wongkewtubularneighborhood}
Let $X \subseteq \R^n$ be a real algebraic variety of codimension $m$ defined by 
polynomials of degree $\le d$.  Then we have the following estimate
\begin{equation} \label{tubularneighborhoodestimate}
   \vol_n (N_\eps(X) \cap [0,1]^n) = O((\eps d)^m)
\end{equation}
where the implied constant depends only on $n$.
\end{lemma}

It can be seen via simple examples that this bound is sharp, up to the implied
constant. For completeness, we reproduce Wongkew's argument below.

\begin{proof}
We proceed by induction on $n$, the case $n=0$ being clear.  All implied
constants depend only on $n$. We assume for convenience that $\eps \leq 1$
(otherwise the desired estimate is clear).

Let $\mathbf H$ be the collection of hyperplanes $H \subseteq \R^n$ given by 
constraining any one of the coordinates to lie in $[-2\eps, 1+2\eps] \cap 
(\eps \Z + \delta)$, where $\delta$ is chosen so that $X$ intersects each $H \in
\mathbf H$ properly (i.e.\ $X \cap H$ has codimension $m$ inside $H$). Such a
$\delta$ exists by Bertini's theorem.  Clearly $\# \mathbf H = O(\eps^{-1})$.  
This set of hyperplanes partitions $\R^n$ into some unbounded components and
some cubes of side length $\eps$. We denote the set of such cubes by $\mathbf
C$.

We call a cube $C\in\mathbf C$ \emph{exceptional} iff $X$ intersects the
interior of $C$ but not its boundary. The number of exceptional cubes is clearly
bounded by $\dim H_0(X)$, which by a result of Milnor \cite{milnor} is bounded
by $d(2d-1)^{n-1} = O(d^n)$.

It is straightforward to check that
\begin{equation} \label{coverbyhyperplaneneighborhoods}
   N_\eps (X) \cap [0,1]^n \subseteq \Bigl[N_{ (1+\sqrt n)\eps }
   \Bigl(X \cap \bigcup_{H \in \mathbf H} H\Bigr) \cap [0,1]^n\Bigr] \cup
   \Bigl[\bigcup_{\begin{smallmatrix} C \in \mathbf C\cr C \text{ exceptional}
   \end{smallmatrix}} N_\eps(C)\Bigr] .
\end{equation}
Indeed, suppose $p \in [0,1]^n$ and $d(p,X) \leq \eps$. There exists $x \in X$ 
with $d(p,x) \leq \eps$, and $x \in C$ for some (closed) cube $C \in \mathbf C$.
If $C$ is exceptional, then $p$ lies in the second term above. If $C$ is not
exceptional, then $X \cap \partial C$ is nonempty. It thus follows that $d (p,
\partial C \cap X) \le \eps + d (x, \partial C \cap X) \le \eps + \eps\sqrt n$,
and so $p$ lies in the first term above.

Now the inclusion \eqref{coverbyhyperplaneneighborhoods} implies the following
inequality on volumes
\begin{multline*}
   \vol_n (N_\eps (X) \cap [0,1]^n) \leq
   \sum_{H \in \mathbf H} 2 (1+\sqrt n) \eps 
   \vol_{n-1} (N_{ (1+\sqrt n)\eps } (X \cap H) \cap H \cap [0,1]^n) \\
 + \sum_{\begin{smallmatrix} C \in \mathbf C \cr C \text{ exceptional}
   \end{smallmatrix}} (3\eps)^n .
\end{multline*}
If $m = n$, then the first term vanishes (each $X \cap H$ is empty by 
assumption), and Milnor's bound on the second term gives the desired result. 
If $m \le n-1$, then we apply the induction hypothesis to the first term and 
Milnor's result to the second term. The result is:
\begin{equation}
   \vol_n (N_\eps (X) \cap [0,1]^n) = O((\eps d)^m + (\eps d)^n) .
\end{equation}
This implies the desired estimate for $\eps d \le 1$, and for $\eps d \ge 1$ the
desired estimate is trivial.
\end{proof}

\section{Lefschetz fibrations on Stein domains}\label{steinLFsec}

We now show how the function $f$ guaranteed to exist by Theorem \ref
{quantitativesteinexistence} gives rise to a Lefschetz fibration.  To be 
precise, we will show that Theorem \ref{quantitativesteinexistence} implies 
Theorem \ref{refinedsteinexistence} and that Theorem \ref{refinedsteinexistence}
implies Theorem \ref{steinexistence}.

\begin{proof}[Proof of Theorem \ref{refinedsteinexistence} from Theorem \ref
{quantitativesteinexistence}]
Fix an embedding $V \hookrightarrow \ol V$ of the Stein domain $V$ into a 
Stein manifold $\ol V$ of the same dimension, and fix an exhausting 
$J$-convex function $\phi: \ol V\to \R$ with $V = \{\phi \le 0\}$.

By Theorem \ref{quantitativesteinexistence}, there exists (for sufficiently
large $k$) a holomorphic function $f: \ol V \to \C$ such that: 
\begin{itemize}
\item
$|f(p)| + k^{-1/2} |df(p) \rst \xi| > \eta$ for $p \in \partial V$.
\item 
$|f(p)| \le e^{\frac 12 k\phi(p)}$ for $p \in \{\phi \le 1\}$.
\end{itemize}
We claim that the bound $|f(p)| \le e^{\frac 12 k\phi(p)}$ implies:
\begin{itemize}
\item
$|df(p) - k \cdot f(p) \cdot d'\phi(p)| = O(k^{1/2} e^{\frac 12 k\phi(p)})$ for 
$p \in V$.
\end{itemize}
To see this, argue as follows. Fix a point $p \in V$ and choose a holomorphic
function $u$ defined in a neighborhood of $p$ such that $\Re u(q) = \phi(q) +
O (d(p,q)^2)$. It follows that $f(q) \cdot e^{-\frac 12 ku(q)} = O(1)$ for 
$d(p,q) = O(k^{-1/2})$, and hence it follows that $d (f \cdot e^{-\frac 12ku})
(p) = O(k^{1/2})$. Expanding the left hand side and using the fact that $du(p)
= 2d'\Re u(p) = 2d'\phi(p)$, the claim follows.

Now we take $\pi := \eta^{-1} \cdot f$, which satisfies the desired properties.
\end{proof}

\begin{proof}[Proof of Theorem \ref{steinexistence} from Theorem \ref
{refinedsteinexistence}]
By Theorem \ref{refinedsteinexistence}, there exists (for sufficiently large 
$k$), a holomorphic function $\pi: V \to \C$ such that:
\begin{itemize}
\item
For $|\pi(p)| \ge 1$, we have $d\log\pi(p) = k \cdot d'\phi(p) + O(k^{1/2})$.
\item
For $|\pi(p)| \le 1$ and $p \in \partial V$, we have $d\pi(p) \rst \xi \ne 0$.
\end{itemize}
Note that these conditions together imply that the critical locus of $\pi$ is 
contained in the interior of $\pi^{-1}(D^2)$.  Both conditions are preserved 
under small perturbations of $\pi$, hence we may perturb $\pi$ so that:
\begin{itemize}
\item
All critical points of $\pi$ on $V$ are nondegenerate and have distinct critical
values.
\end{itemize}
Indeed, the existence of such a perturbation follows from the standard fact that
global holomorphic functions on any Stein manifold $\ol V$ generate 
$\OO_{\ol V}$ and $\Omega^1_{\ol V}$ at every point (this follows from Cartan's 
Theorems A and B, or by properly embedding $\ol V$ in $\C^N$).

Now $\pi: \pi^{-1}(D^2) \to D^2$ is a Stein Lefschetz fibration, so it suffices 
to construct a deformation of Stein domains from $V$ to $\pi^{-1}(D^2)^\sm$.  
Let $g: \R_{<0} \to \R$ satisfy $g'>0$, $g''>0$, and $\lim_{x \to 0^-} g(x) = 
\infty$.  Consider the family $\{ \pi^{-1}(D^2_r) \}_{1 \le r < \infty}$, and 
consider its smoothing $\{r^{-3} g(|\pi|^2-r^2) + g(\phi) \leq M\}_{1 \le r < 
\infty}$ for some large $M < \infty$.  Since $\pi^{-1}(D^2_r)$ is cut out by 
the inequalities $\phi \le 0$ and $\Re\log\pi \le \log r$, this smoothing gives
the desired deformation as long as for every point $p\in V$ with $|\pi(p)| 
\ge 1$, the differentials $d\phi(p)$ and $\Re d\log\pi(p)$ are either linearly 
independent or positively proportional.  Since $d\log\pi(p) = k \cdot d'\phi(p) 
+ O(k^{1/2})$, this condition is clearly satisfied for sufficiently large $k$.
\end{proof}

\section{Lefschetz fibrations on Weinstein domains} \label{weinsteinLFsec}

We now show how the existence of Lefschetz fibrations on Stein domains (Theorem
\ref{steinexistence}) may be used to deduce the same for Weinstein domains
(Theorem \ref{weinsteinexistence}). For this implication, we use the 
result of Cieliebak--Eliashberg \cite[Theorem 1.1(a)]{cieliebakeliashberg} that 
every Weinstein domain may be deformed to carry a compatible Stein structure.  
The main step (Proposition \ref{steinLFtoweinsteinLF}) is thus to show that for 
any Stein Lefschetz fibration $\pi: V \to D^2$, there exists an abstract 
Weinstein Lefschetz fibration whose total space is deformation equivalent to 
$V^\sm$.

\subsection{From Stein structures to Weinstein structures}

We give a very brief review of the relationship between Stein and Weinstein
structures (for a complete treatment, the reader may consult \cite[\S 1]
{cieliebakeliashberg}). Let $(V, \phi)$ be a pair consisting of a Stein domain
$V$ and a smooth $J$-convex function $\phi: V \to \R$ with $\partial V = \{\phi 
= 0\}$ as a regular level set. If $\phi$ is Morse (which can be achieved by
small perturbation), then it induces the structure of a Weinstein domain on $V$,
namely taking the $1$-form $\lambda_\phi := -J^*d\phi$ and the function $\phi$
itself. This Weinstein domain is denoted $\W (V,\phi)$. For any deformation of
Stein domains $(V_t, \phi_t)_{t \in [0,1]}$ where every $\phi_t$ is generalized
Morse (any $\{\phi_t\}_{t\in[0,1]}$ may be perturbed to satisfy this condition),
the associated family $\W(V_t,\phi_t)_{t\in[0,1]}$ is a deformation of Weinstein
domains. In particular, the deformation class of $\W (V,\phi)$ is independent of
$\phi$, so we may denote it by $\W(V)$. Now a decisive result is the following
(we state a simplified version which is sufficient for our purpose).

\begin{theorem}[{Cieliebak--Eliashberg \cite[Theorem 1.1(a)]
{cieliebakeliashberg}}] \label{steinweinstein}
Every deformation class of Weinstein domain is of the form $\W(V)$ for a Stein 
domain $V$.
\end{theorem}

\subsection{From Stein Lefschetz fibrations to abstract Weinstein Lefschetz
fibrations}

Theorem \ref{weinsteinexistence} follows from Theorem \ref{steinexistence}, 
Theorem \ref{steinweinstein}, and the following proposition.

\begin{proposition} \label{steinLFtoweinsteinLF}
Let $\pi: V \to D^2$ be a Stein Lefschetz fibration. There exists an abstract
Weinstein Lefschetz fibration $W = (W_0;  L_1, \ldots, L_m)$ whose total space 
$|W|$ is deformation equivalent to $\W(V^\sm)$.
\end{proposition}

The abstract Weinstein Lefschetz fibration associated to a Stein Lefschetz 
fibration may be described as follows. The ``central fiber'' $W_0$ is the 
Weinstein domain associated to a regular fiber $\pi^{-1}(p)$ of $\pi: V \to
D^2$, and the ``vanishing cycles'' $L_1, \ldots, L_m$ are the images of the
critical points of $\pi$ under symplectic parallel transport along a set of
disjoint paths from the critical values of $\pi$ to the regular value $p$. Hence
the content of the proposition is that (as a Weinstein manifold) $V^\sm$ may be 
described as a small product neighborhood of a regular fiber with Weinstein
handles attached along the vanishing cycles.

We now give a detailed definition of the total space of an abstract Weinstein 
Lefschetz fibration.

\begin{definition} \label{totalspacedef}
Let $W = ((W_0, \lambda_0, \phi_0) ; L_1, \ldots, L_m)$ be an abstract Weinstein
Lefschetz fibration. Its \emph{total space} $|W|$ is defined as follows. We
equip $W_0 \times \C$ with the Liouville form $\lambda_0 - J^*d(\frac12 |z|^2)$ 
and the Morse function $\phi_0 + |z|^2$ for which the resulting Liouville 
vector field $X_{\lambda_0}+\frac 12(x\der_x{}+y\der_y{})$ is gradient-like. 
Fix Legendrian lifts $\Lambda_j
\subseteq (W_0 \times S^1, \lambda_0 + N d\theta)$ of the exact Lagrangians
$L_j \subseteq W_0$ such that $\Lambda_j$ projects to a small interval around
$2\pi j/m \in S^1$ (here we choose $N < \infty$ sufficiently large so that these
intervals are disjoint). Now the embedding $S^1 \hookrightarrow \C$ as the
circle of radius $\sqrt N$ pulls back the Liouville form $-J^*d(\frac 12 |z|^2)$ to
the contact form $N d\theta$. Hence we may think of $\Lambda_j$ as lying inside 
$W_0 \times \C$ as a Legendrian on the level set $\{|z| = \sqrt N\}$. The downward
Liouville flow applied to $\Lambda_j$ gives rise to a map $\Lambda_j \times
\R_{\geq 0} \to W_0 \times \C$, which intersects the level set $\{\phi_0 + |z|^2
= 0\}$ in a Legendrian $\Lambda_j'$ (here we choose $N < \infty$ so that the
projection of $\{\phi_0 + |z|^2 \leq 0\}$ to $\C$ is contained inside the disk
of radius $\sqrt N$). The total space $|W|$ is defined as the result of attaching
Weinstein handles (\cite{weinstein}) to the Weinstein domain $\{\phi_0 + |z|^2
\leq 0\}$ along the Legendrians $\Lambda_j'$ (marked via the maps $S^{n-1} \to
L_j \xrightarrow\sim \Lambda_j \xrightarrow\sim \Lambda_j'$). It is easy to see 
that $|W|$ is well-defined up to canonical deformation (we will remark in detail
on the well-definedness of Weinstein handle attachment in Lemma \ref
{elementarycobordismisweinsteinhandle}).
\end{definition}

We now introduce a variant of the above construction, which will be used in the 
proof of Proposition \ref{steinLFtoweinsteinLF}.

\begin{definition} \label{steinLFplusweinsteinhandles}
Let $W = (\pi: V \to D^2, \phi, g; L_1, \ldots, L_m)$ consist of a Stein 
Lefschetz fibration $\pi: V \to D^2$, a $J$-convex function $\phi: V \to \R$ 
with $\partial_h V = \{\phi = 0\}$ as a regular level set, a function $g: 
\R_{<0} \to \R$ with $g>0$, $g'>0$, $g''>0$, and $\lim_{x \to 0^-} g(x) = 
\infty$, and a collection of exact parameterized Lagrangian (with respect to 
$\lambda_{g(\phi)}$) spheres $L_j \subseteq V_{p_j} := \pi^{-1}(p_j)$ for 
distinct points $p_1, \ldots, p_m \in S^1 = \partial D^2$, ordered 
counterclockwise. We define its \emph{total space} $|W|$ as follows. We
consider the $J$-convex function $\eps g(\phi) + \frac 12 |\pi|^2$ on $V$. The
induced contact form on $\pi^{-1} (\partial D^2)$ may be written as $\eps
\lambda_{g(\phi)} + d\theta$. Let us center the $S^1$-coordinate at $p_j \in
S^1$, rescale it by $\eps^{-1}$, and rescale the contact form by $\eps^{-1}$. In
the limit $\eps \to 0$, this rescaling of $\pi^{-1} (\partial D^2)$ converges to
the contact manifold $(V_{p_j} \times \R, \lambda_{g(\phi)} + dt)$. In $V_{p_j} 
\times \R$, there is a unique (up to translation) Legendrian $\Lambda_j$
projecting to $L_j$. During the deformation of $V_{p_j} \times \R$ back to
$\pi^{-1} (\partial D^2)$ for small $\eps>0$, there clearly exists a
simultaneous Legendrian isotopy $\Lambda_j^\eps \subseteq \pi^{-1} (\partial 
D^2)$ starting at $\Lambda_j^0 = \Lambda_j$. Now the downward Liouville applied 
to $\Lambda_j^\eps$ intersects $\{\eps g(\phi) + \frac 12 (|\pi|^2 - 1) = 0\}$
in a Legendrian $\Lambda_j^{\eps\prime}$. The total space $|W|$ is defined as
the result of attaching Weinstein handles to the Weinstein domain $\{\eps
g(\phi) + \frac 12 (|\pi|^2-1) \leq 0\}$ along these Legendrians. This total
space is independent of the choice of sufficiently small $\eps > 0$ and the 
family $\{\Lambda_i^\eps\}_{\eps \geq 0}$ up to canonical deformation.
\end{definition}

Definition \ref{steinLFplusweinsteinhandles} reduces to Definition 
\ref{totalspacedef} in the special case of a product fibration, in the sense 
that there is a canonical deformation equivalence
\begin{equation} \label{twototalspaces}
\bigl| (V_0 \times D^2 \to D^2, \phi_0, g; L_1 \times \{\alpha_1\}, \ldots, L_m 
\times \{\alpha_m\}) \bigr| = \bigl| (\W(V_0, g(\phi_0)); L_1, \ldots, L_m)
\bigr| .
\end{equation}
where $\phi_0: V_0 \to \R$ is $J$-convex with $\partial V_0 = \{\phi_0 = 0\}$ 
as a regular level set, $L_1, \ldots, L_m \subseteq V_0$ are exact parameterized
Lagrangian spheres with respect to $\lambda_{g(\phi_0)}$, and $\alpha_1, \ldots,
\alpha_m \in S^1 = \partial D^2$ are ordered counterclockwise.  The right hand 
side of \eqref{twototalspaces} is a slight abuse of notation, as we should 
really write $\W(\{g(\phi_0) \leq M\}, g(\phi_0))$ for sufficiently large $M$.

\begin{proof}[Proof of Proposition \ref{steinLFtoweinsteinLF}]
We assume that $0 \in D^2$ is a regular value of $\pi$ and that each critical
value of $\pi$ has a distinct complex argument (this may be achieved by
post-composing $\pi$ with a generic Schwarz biholomorphism $D^2 \to D^2$).

Fix a smooth $J$-convex function $\phi: V \to \R$ with $\partial_h V = \{\phi = 
0\}$ as a regular level set (as is guaranteed to exist by Definition \ref
{steinLF}). We let $V_0 := \pi^{-1}(0)$ denote the central fiber, and we assume 
that $\phi_0 := \phi \rst{ V_0 }$ is Morse (this can be achieved by a small
perturbation of $\phi$).

By Lemma \ref{makecomplete} below, there exists a smooth function $g: \R_{<0} 
\to \R$ satisfying $g>0$, $g'>0$, $g''>0$, and $\lim_{x \to 0^-} g(x) = \infty$,
such that the symplectic connection on $\pi: V \setminus \partial_h V \to D^2$ 
induced by $\omega_{g(\phi)}$ is \emph{complete}. Fix one such $g$.

We consider parallel transport along radial paths in $D^2$ with respect to the 
symplectic connection induced by $\omega_{g(\phi)}$. Under this parallel 
transport, each critical point of $\pi$ sweeps out a Lagrangian disk called a 
\emph{Lefschetz thimble} (to see this, apply the stable manifold theorem to the 
Hamiltonian vector field $X_{\Im \log\pi}$, and recall that the critical values 
of $\pi$ have distinct complex arguments). The fiber over $0 \in D^2$ of a 
Lefschetz thimble is an exact Lagrangian sphere called 
a \emph{vanishing cycle}. Let $L_1, \ldots, L_m \subseteq V_0$ denote the
vanishing cycles of all the critical points of $\pi$, ordered by angle. As
stable manifolds of the vector field $X_{\Im \log\pi}$, they come equipped with 
parameterizations $S^{n-1} \to L_j$, which are well-defined in $\Diff (S^{n-1}, 
L_j) / O(n)$ up to contractible choice.

Now $W := (\W (V_0, g(\phi_0)); L_1, \ldots, L_m)$ is an abstract Weinstein
Lefschetz fibration, and it remains to show that its total space $|W|$ is 
deformation equivalent to $\W (V^\sm)$.

We consider the $J$-convex function $\eps g(\phi) + h(\frac {|\pi|} \delta)$ on 
$V \setminus \partial_h V$ for small $\eps, \delta > 0$, where
\begin{equation}
   h(r) := \begin{cases}
   \log r & r \geq 1\cr
   \frac12 (r^2-1) & r \leq 1.
\end{cases}
\end{equation}
We claim that for $(\eps, \delta) \to (0,0)$, the sublevel set
\begin{equation} \label{sublevelset}
   \Bigl\{\eps g(\phi) + h\Bigl(\frac {|\pi|} \delta\Bigr)
   \leq \log \frac 1\delta \Bigr\} \subseteq V
\end{equation}
is deformation equivalent to $V^\sm$. Indeed, consider the 
$\leq \log \frac 1\delta$ sublevel set of the linear interpolation between 
$\eps g(\phi) + h(\frac {|\pi|} \delta)$ and $\eps g(\phi) + \eps g(|\pi|^2-1)$.
As $(\eps,\delta) \to (0,0)$, the boundary of this deformation stays arbitrarily
close to $\partial V$, and the critical locus of the linear interpolation stays 
away from $\partial V$ (note that this critical locus is always contained in the
fiberwise critical locus of $\phi$). Thus \eqref{sublevelset} is deformation
equivalent to $V^\sm$ as claimed.

As $(\eps, \delta) \to (0,0)$, the critical points of $\eps g(\phi) + h(\frac
{|\pi|} \delta)$ over $D^2 \setminus D^2_\delta$ are in bijective correspondence
with $\crit(\pi)$ (note that the critical locus is contained in the fiberwise
critical locus of $\phi$). Over $D^2 \setminus D^2_\delta$, the stable manifolds
of these critical points approach the Lefschetz thimbles as $\eps \to 0$ and
$\delta>0$ is fixed. Indeed, $h$ is harmonic over $D^2 \setminus D^2_\delta$,
and hence the Liouville vector field of $\eps g(\phi) + h(\frac {|\pi|} \delta)$
is given by $X_{g(\phi)} + \eps^{-1}X_{\Im\log\pi}$ over $D^2 \setminus 
D^2_\delta$, where $X_{g(\phi)}$ is the Liouville vector field of $g(\phi)$ and
$X_{\Im\log\pi}$ is the Hamiltonian vector field with respect to
$\omega_{g(\phi)}$ of $\Im\log\pi$.

Let us denote by $\bar\Lambda_j^{\eps,\delta} \subseteq \pi^{-1}
(\partial D^2_\delta)$ the intersections of the stable manifolds of $\eps
g(\phi) + h (\frac {|\pi|} \delta)$ with $\pi^{-1}(\partial D^2_\delta)$. Thus
$\bar\Lambda_j^{\eps,\delta}$ is Legendrian with respect to the contact form
$\eps\lambda_{g(\phi)} + d\theta$. Denote by $L_j^\delta$ the intersections of
the Lefschetz thimbles with $\pi^{-1} (\partial D^2_\delta)$. Thus as $\eps \to 
0$ and $\delta>0$ is fixed, we have $\bar\Lambda_j^{\eps,\delta} \to
L_j^\delta$ in $C^\infty$. Now we claim that $\bar\Lambda_j^{\eps,\delta}$ is 
in fact (Legendrian isotopic to) the Legendrian lift $\Lambda_j^{\eps,\delta}$ 
of $L_j^\delta$ (as in Definition \ref{steinLFplusweinsteinhandles}) for 
sufficiently small $\eps>0$. In the rescaled limit as $\eps\to 0$, the  
projection of $\bar\Lambda_j^{\eps,\delta}$ to $V_{p_j}$ approaches 
$L_j^\delta$ in $C^\infty$, and this is enough to show that it converges (up to 
translation) to $\Lambda_j^{0,\delta}$ as $\eps \to 0$.  Hence the claim is 
valid, so we conclude that $\W(V^\sm)$ is deformation equivalent to
\begin{equation} \label{deltatotalspace}
   \bigl|(\pi: \pi^{-1} (D^2_\delta) \to D^2_\delta; 
   L_1^\delta, \ldots, L_m^\delta)\bigr| .
\end{equation}
We have used Lemma \ref{elementarycobordismisweinsteinhandle} below to show that
the Weinstein cobordism $\{0\leq\eps g(\phi) + h(\frac {|\pi|} \delta) \leq 
\log \frac 1\delta\}$ is a Weinstein handle attachment.

In the limit $\delta \to 0$, rescaling $D^2_\delta$ to $D^2$, clearly 
\eqref{deltatotalspace} converges to
\begin{equation}
   \bigl|(V_0 \times D^2 \to D^2; L_1 \times \{\alpha_1\}, \ldots,
   L_m \times \{\alpha_m\})\bigr|
\end{equation}
where $\alpha_j \in S^1 = \partial D^2$ are the angles of the critical points of
$\pi$. Hence using \eqref{twototalspaces}, we have shown the desired deformation
equivalence between $\W(V^\sm)$ and $|(\W(V_0, g(\phi_0)); L_1, \ldots, L_m)|$.
\end{proof}

\begin{lemma} \label{makecomplete}
Let $\pi: V \to D^2$ be a Stein Lefschetz fibration, and let $\phi: V \to \R$ be
$J$-convex with $\partial_h V = \{\phi = 0\}$ as a regular level set. There
exists a smooth function $g: \R_{<0} \to \R$ satisfying $g'>0$, $g''>0$, and 
$\lim_{x \to 0^-} g(x) = \infty$, such that the symplectic connection on 
$\pi: V \setminus \partial_h V \to D^2$ induced by $\omega_{g(\phi)}$ is 
\emph{complete}, in the sense that parallel transport along a smooth path in 
the base $D^2$ gives rise to a diffeomorphism between the corresponding fibers 
(away from the critical points of $\pi$).
\end{lemma}

\begin{proof}
We will in fact show that there exists a natural contractible family of 
functions $g$ which satisfy the desired conclusion for all $(V,\phi)$.

Let $g: \R_{<0} \to \R$ satisfy $g'>0$, $g''>0$, and $\lim_{x \to 0^-} g(x) = 
\infty$. Let $\perp_\phi$ (resp.\ $\perp_{g(\phi)}$) denote orthogonal
complement with respect to $\omega_\phi$ (resp.\ $\omega_{g(\phi)}$), so the
horizontal distribution of the symplectic connection induced by
$\omega_{g(\phi)}$ is $(\ker d\pi)^{\perp_{g(\phi)}}$.

Our first goal is to show that in a neighborhood of $\partial_hV$, every 
horizontal vector field $X$ satisfies
\begin{equation}\label{boundonXphi}
   |X \phi| = O \left(
   \frac{ g'(\phi) }{ g''(\phi) }\right) \cdot
   |\pi_*X|
\end{equation}
as long as $\frac{ g'(\phi) }{ g''(\phi) }$ sufficiently small.  Note that in a
neighborhood of $\partial_h V$, there is a direct sum decomposition
\begin{align}
   TV = (\ker d\pi \cap \ker d'\phi) &\oplus (\ker d\pi \cap \ker d'\phi)^{\perp_\phi} \cap \ker d\pi\cr
   &\oplus (\ker d\pi \cap \ker d'\phi)^{\perp_\phi} \cap \ker d'\phi
\end{align}
into subspaces of real dimension $2n-4$, $2$, $2$, respectively.  Now suppose
that $X = X_1 \oplus X_2 \oplus X_3 \in TV$ is horizontal, i.e.\ $X 
\perp_{g(\phi)} \ker d\pi$.  Note the explicit form
\begin{equation}
   \omega_{g(\phi)} = g'(\phi) \cdot \omega_\phi
 + g''(\phi) \cdot id'\phi \wedge d''\phi .
\end{equation}
We may choose a vector $v \in (\ker d\pi \cap \ker d'\phi)^{\perp_\phi} \cap 
\ker d\pi$ with $|v|_{g_\phi}=1$ such that $|(id'\phi \wedge
d''\phi)(v,X_2)|\asymp|X_2|_{g_\phi}$ (where $g_\phi$ denotes the metric induced
by $\phi$).  Now since $v \in \ker d\pi$, it pairs to zero with $X$ under 
$\omega_{g(\phi)}$, so we have
\begin{equation}
 0 = g'(\phi) \cdot \omega_\phi(v,X_2+X_3) + g''(\phi) \cdot (id'\phi \wedge d''\phi)(v,X_2)
\end{equation}
It follows from this that $|X_2|_{g_\phi} = O\bigl(\frac{ g'(\phi) }{ g''(\phi) 
}\bigr) \cdot |X_3|_{g_\phi}$ for $\frac{ g'(\phi) }{ g''(\phi) }$ sufficiently 
small.  This implies the desired estimate \eqref{boundonXphi} since $|\pi_* X|
\asymp |X_3|_{g_\phi}$ and $|X\phi| \asymp |X_2|_{g_\phi}$.

It now follows that the connection is complete as long as
\begin{equation} \label{completnesscondition}
   \limsup_{x\to 0^-} \frac {g'(x)} {|x| g''(x)} < \infty .
\end{equation}
Indeed, by \eqref{boundonXphi} this condition guarantees that the derivative of
$\log(-\phi)$ is bounded along the horizontal lift of a smooth curve in the 
base $D^2$.

We now just need to exhibit a function $g: \R_{<0} \to \R$ satisfying $g'>0$, 
$g''>0$, $\lim_{x \to 0^-} g(x) = \infty$, and \eqref{completnesscondition}, 
which we may write as
\begin{equation}
   \liminf_{x\to 0^-} \; (\log g'(x))'|x|>0.
\end{equation}
For example, we may take
\begin{equation}
   g(x) := \int_{-\infty}^x e^{-t^2-t^{-1}}\, dt .
\end{equation}
Moreover, the space of such functions is contractible, since 
the map $g \mapsto (g(-1), \log g')$ gives a bijection with a convex set.
\end{proof}

\subsection{Uniqueness of Weinstein handle attachment}

We record here a proof of the fact that an elementary Weinstein cobordism is 
``the same'' as a Weinstein handle attachment (the precise statement is Lemma 
\ref{elementarycobordismisweinsteinhandle}), as was used in the proof of 
Proposition \ref{steinLFtoweinsteinLF}. We were unable to find a precise 
reference for this standard fact, though it is of course implicit in Weinstein's
original paper \cite{weinstein}, as well as in Cieliebak--Eliashberg \cite
{cieliebakeliashberg}.

Recall that a Weinstein cobordism $(W, \lambda, \phi)$ is called \emph
{elementary} iff there is no trajectory of $X = X_\lambda$ between any two
critical points. For a critical point $p \in W$, we denote by $T_p^\pm W$ the 
positive/negative eigenspaces of $d_pX: T_pW \to T_pW$, and we denote the stable
manifold by $W_p^-$. For an elementary cobordism, each stable manifold $W_p^-$ 
intersects the negative boundary $\partial_- W$ in an isotropic sphere 
$\Lambda_p \subseteq \partial_- W$; note that $\Lambda_p = (W_p^- \setminus p) 
/ \R$ via the Liouville flow.  A choice of exponential coordinates $\exp_p: 
T_p^- W \to W_p^-$ and a small sphere centered at zero in $T^-_pW$ determines 
a diffeomorphism $(W_p^- \setminus p) / \R = (T_p^- W \setminus 0) / \R$.  We 
thus obtain a diffeomorphism $\rho_p \in \Diff ((T_p^-W \setminus 0) / \R, 
\Lambda_p)$ which is well-defined up to contractible choice.

\begin{lemma} \label{elementarycobordismisweinsteinhandle}
Let $(Y^{2n-1}, \lambda)$ be a contact manifold with contact form, let 
$\Lambda_1, \ldots, \Lambda_m \subseteq Y$ be disjoint Legendrian spheres, and
let $\sigma_j \in \Diff (S^{n-1}, \Lambda_j) / O(n)$. The following space is 
weakly contractible:
\begin{equation*}
\left\{ \parbox{4.8in}{
Elementary Weinstein cobordism $(W^{2n}, \lambda, \phi)$ with critical points
$p_j$ and stable manifolds $V_{p_j}$.\\
Isomorphism $i: (\partial_-W, \lambda) \xrightarrow \sim (Y, \lambda)$ sending 
$V_{p_j} \cap \partial_-W$ to $\Lambda_j$.\\
Path $q$ between $\sigma_j$ and the image of $\rho_{p_j}$ in 
$\Diff (S^{n-1}, \Lambda_j) / O(n)$.
}\right\}
\end{equation*}
in the sense that for all $k \geq 0$, any family of such objects $(W,i,q)$ over 
$\partial D^k$ can be extended to a family over $D^k$.
\end{lemma}

There is also a version of Lemma \ref{elementarycobordismisweinsteinhandle} for 
any critical points of any index, though it is more complicated to state since 
subcritical handle attachment requires an additional piece of data (a framing 
of the symplectic normal bundle of the attaching sphere). In this paper, we only
need the case of critical handle attachment, so we omit the more general 
statement and its proof. We thank Ya.\ Eliashberg for useful discussions 
regarding the proof.

\begin{proof}
Let a family over $\partial D^k$ be given ($k \geq 0$).

We first equip the family with local Darboux charts near the critical points,
and homotope it so that the Liouville vector field coincides with a certain
standard model in these charts. We phrase this part of the argument as if there 
is just a single triple $(W,i,q)$ and a single critical point $p$, but it is
clear that each step also works in families and for multiple critical points.
The details are as follows.

Fix a local symplectomorphism (Darboux chart) $\exp_p: (T_pW,0) \to (W,p)$ 
whose derivative at zero is the identity.  On the symplectic vector space 
$T_pW$, the vector field $d_pX: T_pW \to T_pW$ is
Liouville (this is just the linearization of the Liouville structure of $W$ near
$p$); it follows that the positive/negative eigenspaces $T_p^\pm W$ of $d_pX$
are Lagrangian \cite[Proposition 11.9]{cieliebakeliashberg}.

We first homotope the function $\phi$ so that
\begin{equation} \label{phistandard}
\exp_p^* \phi = \phi_\std \quad \text{near zero}
\end{equation}
where $\phi_\std: T_pW \to \R$ is given by $\phi_\std (v) := |v^+|^2 - |v^-|^2$.
Here we fix positive definite quadratic forms on $T_p^\pm W$ such that the 
Liouville vector field $\exp_p^* X$ is gradient-like for $\phi_\std$ near zero.
Note that the space of such quadratic forms is clearly open and convex, and it 
is seen to be non-empty by considering quadratic forms which are diagonal with 
respect to a basis which puts $d_pX$ into Jordan normal form.  Now we consider 
the homotopy $\{ 
\phi + (\phi_\std - \phi)t\chi \}_{t \in [0,1]}$ for some smooth compactly 
supported cutoff function $\chi: T_pW \to [0,1]$ which equals $1$ in a 
neighborhood of zero.  Its differential equals $(1-t\chi)d\phi + t\chi 
d\phi_\std + t(\phi_\std-\phi)d\chi$.  We have $\phi_\std-\phi=O(|v|^2)$, so to 
ensure that $\exp_p^* X$ is gradient-like throughout the homotopy, it suffices 
to choose $\chi$ so that $|d\chi|$ is much smaller than $|v|^{-1}$. Such a 
cutoff function exists (supported in any given neighborhood of zero) since 
$\int_0^1 r^{-1}dr$ diverges.  Thus we have achieved \eqref{phistandard}.

We next homotope the Liouville vector field $X$ so that
\begin{equation} \label{Xstandard}
\exp_p^* X = X_\std \quad \text{near zero}
\end{equation}
where $X_\std: T_pW \to T_pW$ acts by $-\id$ on $T_p^-W$ and by $2\id$ on 
$T_p^+W$ (observe that this is indeed a Liouville vector field).  Note that both
vector fields $\exp_p^* X$ and $X_\std$ are gradient-like with respect to 
$\exp_p^* \phi = \phi_\std$ near zero.  Write the Liouville form for 
$\exp_p^* X$ as $\lambda$, write the Liouville form for $X_\std$ as 
$\lambda_\std$, and write $\lambda_\std - \lambda = df$ for a function $f$ 
vanishing at zero.  We consider the homotopy $\{\lambda + d (t\chi f)\}_{t \in 
[0,1]}$ for $\chi$ as above. We may write this as
$(1-t\chi)\lambda + t\chi\lambda_\std + tfd\chi$. We have $f = O(|v|^2)$, so in 
order to guarantee that the resulting Liouville vector field remains 
gradient-like for $\exp_p^*\phi = \phi_\std$, it is again enough to choose 
$\chi$ so that $|d\chi|$ is much smaller than $|v|^{-1}$, which exists as 
before.  This achieves \eqref{Xstandard}.

We have now homotoped $X$ and $\phi$ near $p$ so that they coincide via the 
chosen Darboux chart $\exp_p: (T_pW,0) \to (W,p)$ with $X_\std$ and $\phi_\std$ 
as above near zero. Since $\exp_p^*X = X_\std$ in a neighborhood of zero, there 
is an induced diffeomorphism $\rho: (T_p^-W \setminus 0) / \R_{>0} \to \Lambda$.

Now $\Diff (S^{n-1}, \Lambda) / O(n)$ classifies vector bundles $V$ along with a
fiberwise diffeomorphism from the sphere bundle $S(V) := (V \setminus 0) /
\R_{>0}$ to $\Lambda$. Hence the data of $q$ determines an extension of the
vector bundle $T_p^-W$ from $\partial D^k$ to $D^k$ (also denoted $T_p^-W$),
an extension of the fiberwise diffeomorphism $\rho: (T_p^-W \setminus 0) / 
\R_{>0} \to \Lambda$ to $D^k$, and an extension of $q$ itself from $\partial
D^k$ to $D^k$. We may also extend $T_p^+W$ from $\partial D^k$ to $D^k$ by
observing that $T_p^+W = (T_p^-W)^*$ over $\partial D^k$ (by virtue of the
symplectic form) and thus defining $T_p^+W := (T_p^-W)^*$ over $D^k$. Hence
$T_pW := T_p^-W \oplus T_p^+W$ is a symplectic vector bundle over $D^k$. We
conclude that it suffices to extend $W$ from $\partial D^k$ to $D^k$ so that it 
has the chosen tangent spaces $T_pW$, has exponential charts $\exp_p$ satisfying
\eqref{phistandard} and \eqref{Xstandard} above, and so that it induces the 
chosen diffeomorphisms $\rho: (T_p^-W \setminus 0) / \R_{>0} \to \Lambda$.

Over any point in $D^k$, we have a co-oriented contact manifold $(T_pW \setminus T_p^+W) 
/ \R$ (quotient by the Liouville flow), and a Legendrian submanifold $(T_p^-W
\setminus 0) / \R_{>0}$ (quotient by dilation, which coincides with the
Liouville flow). Over any point in $\partial D^k$, the Liouville flow on $W$
determines a germ of co-orientation preserving contactomorphism $\tilde\rho$ between a neighborhood of
this Legendrian submanifold and a neighborhood of $\Lambda_p \subseteq Y$,
restricting to $\rho$. Conversely, a neighborhood of $W_p^-$ in $W$ is
determined by $T_pW = T_p^-W \oplus T_p^+W$ and the germ of co-orientation preserving contactomorphism
$\tilde\rho$. Note that $W$ always deforms down to a neighborhood of $\partial_-
W \cup W_p^-$. Thus it suffices to extend $\tilde\rho$ from $\partial D^k$ to
$D^k$ such that it restricts to $\rho$ (such an extension determines for us an
extension of $W$ from $\partial D^k$ to $D^k$).

To show that $\tilde\rho$ extends to $D^k$, it suffices to show that for any 
closed manifold $M$, the restriction map from germs of co-orientation preserving contactomorphisms of 
$J^1M$ mapping the zero section to itself to diffeomorphisms of $M$ is a weak 
homotopy equivalence (we will apply this to $M = S^{n-1}$). Equivalently, it 
suffices to show that the space of germs of co-orientation preserving contactomorphisms of $J^1M$ fixing 
the zero section pointwise is weakly contractible. Write $J^1M = T^*M \times \R$
with contact form $\lambda - ds$, and write $h_t$ for the flow of the contact 
vector field $X_\lambda + s\der_s{}$.
Fix any germ of co-orientation preserving contactomorphism $f:J^1M\to J^1M$
fixing the zero section pointwise, and we will define a canonical path from $f$
to the identity (clearly this is enough).  We first consider the limit as $t
\to \infty$ of the conjugation $h_t \circ f \circ h_t^{-1}$, which is nothing
other than the vertical projection of the derivative of $f$ along the zero 
section.  We are thus reduced to
considering a co-orientation preserving contactomorphism $f_0:J^1M \to J^1M$
which is a linear map of bundles over $M$.  Now a general such linear map has
the form
\begin{equation}
(\alpha,g)\mapsto(A\alpha+Bg,C\alpha+Dg)
\end{equation}
where $A: M \to \operatorname{End}(T^*M)$, $B: M \to T^*M$, $C: M \to TM$, $D: M
\to \R$ are sections over $M$.  As a contactomorphism, $f_0$ preserves the
Legendrian sections $(dg,g)$ of $J^1M$ over $M$, which means that
\begin{equation}
   A(dg) + gB = d(Cg) + gdD+Ddg
\end{equation}
for all functions $g: M \to \R$.  Since $d(Cg)$ is the only second-order term, 
we conclude that $C \equiv 0$.  Comparing first-order terms shows that $A = D
\cdot \id$, and finally we may solve for $B = dD$.  Thus $f_0: J^1M \to J^1M$ 
is given by $(\alpha, g) \mapsto (D \cdot
\alpha + g \cdot dD, D \cdot g)$ for some function $D: M \to \R$.  Since $f_0$
is a diffeomorphism, $D$ is non-vanishing, and since $f_0$ is co-orientation 
preserving, $D>0$ everywhere.  Finally, we may connect $f_0$ to the identity
using the obvious linear homotopy from $D$ to the constant function $1$.
\end{proof}

\bibliographystyle{amsalpha}
\bibliography{lefschetz}

\end{document}